\documentclass[review,hidelinks,onefignum,onetabnum]
{siamart250211}

\usepackage{lipsum}
\usepackage{amsfonts}
\usepackage{graphicx}
\usepackage{epstopdf}
\usepackage{algorithmic}
\usepackage{enumerate}
\usepackage{booktabs}
\usepackage{mathtools}
\usepackage{mathrsfs}
\usepackage{multirow}
\usepackage{calc}
\newlength{\mylen}
\newcommand{\tabincell}[2]{\begin{tabular}{@{}#1@{}}#2\end{tabular}}

\ifpdf
\DeclareGraphicsExtensions{.eps,.pdf,.png,.jpg}
\else
\DeclareGraphicsExtensions{.eps}
\fi

\newsiamremark{remark}{Remark}
\newsiamremark{hypothesis}{Hypothesis}
\crefname{hypothesis}{Hypothesis}{Hypotheses}
\newsiamthm{claim}{Claim}
\newsiamremark{fact}{Fact}
\crefname{fact}{Fact}{Facts}

\newcommand{\R}{\mathbb{R}}
\newcommand{\Rp}{\mathbb{R}_{+}}
\newcommand{\Rpp}{\mathbb{R}_{++}}
\newcommand{\Rsig}{\mathbb{R}^{\boldsymbol \sigma}}
\newcommand{\Rsigp}{\mathbb{R}^{\boldsymbol \sigma}_{+}}
\newcommand{\Rsigpp}{\mathbb{R}^{\boldsymbol \sigma}_{++}}
\newcommand{\bssig}{{\boldsymbol \sigma}}
\newcommand{\sigp}{({\boldsymbol \sigma}, {\bf p})}
\newcommand{\barf}[1]{{\bf {\bar #1}}}
\newtheorem{assumption}[theorem]{Assumption}
\newtheorem{example}[theorem]{Example}
\newcommand{\diag}{\mathrm{diag}}

\headers{Variational characterization and line search Newton-Noda method}{
	J. Xu, X. Bai{\and} D.-H. Li 
}

\title{A Variational Characterization and
	A Line Search Newton-Noda Method for
	the unifying spectral
	problem of nonnegative tensors
	\thanks{Submitted to the editors \today.
		\funding{This work was supported partly by the the National Natural Science Foundation of China under grants 12271187 and 12301396.}}
	}

\author{
	Jiefeng Xu\thanks{Department of Applied Mathematics, 
    the Hong Kong Polytechnic University, 
    Hong Kong, 
    People’s Republic of China (\email{jiefeng.xu@polyu.edu.hk}).}
	\and Xueli Bai\thanks{School of Mathematics and Statistics, Guangdong University of Foreign Studies, Guangzhou, 
    Guangdong, People’s Republic of China (\email{202210047@gdufs.edu.cn}).}
	\and Dong-Hui Li\thanks{School of Mathematical Sciences, South China Normal University, Guangzhou,
Guangdong, People’s Republic of China (\email{lidonghui@m.scnu.edu.cn}).}
	}

\usepackage{amsopn}

\ifpdf
\hypersetup{
	pdftitle={ls-NNM},
	pdfauthor={Xu-Bai-Li
	}
}
\fi

\begin{document}
	
	\maketitle
	
	\begin{abstract}
	We study the general $(\boldsymbol{\sigma},\mathbf{p})$-eigenvalue problem of nonnegative tensors introduced by A. Gautier,
    F. Tudisco, and M. Hein [SIAM J. Matrix Anal. Appl., 40 (2019), pp. 1206--1231], which unifies several well-studied tensor eigenvalue and singular value problems.
	First, we propose an alternative min-max Collatz--Wielandt formula for the $(\boldsymbol{\sigma},\mathbf{p})$-spectral radius,
	which bypasses the auxiliary multihomogeneous mapping employed in that work.
	This variational characterization both recovers several classical results
	and admits a natural convex reformulation.
	It arises from an alternative approach that directly connects the $(\boldsymbol{\sigma},\mathbf{p})$-spectral problem  to a class of convex programs.	
		
	We then develop and analyze a line search Newton-Noda method (LS-NNM) for computing the positive $(\boldsymbol{\sigma},\mathbf{p})$-eigenpair of nonnegative tensors.
    The proposed method integrates Newton method with Noda iteration.
	The Newton equation is derived from an equivalent nonlinear system, while the eigenvalue sequence is updated	by the strategy of the Noda iteration and its variants.
	To ensure global convergence, we introduce a positivity-preserving line search procedure
	based on an equivalent constrained optimization problem.
	The global and quadratic convergence of LS-NNM are established for the class of $(\bssig, {\bf p})$-spectral problem that admits 
    a unique positive $(\bssig, {\bf p})$-eigenpair,
    as guaranteed by the Perron-Frobenius theorem. 
    Finally, numerical experiments are conducted to illustrate the performance of LS-NNM.
	\end{abstract}
	
	
	\section{Introduction}
	The spectral problem of higher-order tensors has important applications in spectral hypergraph theory, higher-order Markov chains, network science, and magnetic resonance imaging; see, e.g., \cite{AT19,BGL17,LN14,NQZ10,NTGH16,QT03,SCR22,TH21,ZZJNH22}. Unlike matrices, higher-order tensors are inherently multi-dimensional, which has led to multiple, distinct definitions of tensor eigenvalues and singular values in the literature.
	
	In this work, we focus on the general $(\boldsymbol{\sigma},\mathbf{p})$-eigenvalue problem
	of nonnegative tensors introduced by Gautier et al.\ \cite{GTH19,GTH23}.
	Here, $\boldsymbol{\sigma}$ denotes a {\emph{shape partition}} of a tensor
	(i.e., a division that divides its modes into disjoint blocks 
	such that all modes within the same block have the same dimension) 
	and $\mathbf{p}$ is a vector of normalization parameters; 
	see the formal definition in section~\ref{pre}.
	This framework unifies several well-known 
	tensor eigenvalue and singular value problems, including: 
	\begin{itemize}
		\item the classical matrix eigenvalue and singular value problems;
		\item the tensor Z-, H- and $\ell^{p_1}$-eigenvalue problems \cite{CPZ09,Lim2005Singular,Qi05EigenTen};
		\item the tensor $\ell^{p_1,\dots,p_m}$-singular value problems \cite{CQZ10,Lim2005Singular}.
	\end{itemize}
	These problems have been extensively studied from both theoretical and algorithmic perspectives. 
	
	\subsection{Collatz–Wielandt Formula}
	A central theme in the theory of nonnegative tensors is the extension of the Perron--Frobenius theorem from matrices to tensors, 
	which concerns the existence and uniqueness of positive eigenpairs. 
	The first systematic Perron--Frobenius theorem for nonnegative tensors was established by
	Chang et al.\ \cite{CPZ08},
	with further results developed by Yang and Yang \cite{YY10,YY11}.
	Advances on the tensor singular value problem were presented in \cite{CQZ10,FGH13,LQ13,YY11Singular}.
	More recently, Gautier et al.\ \cite{GTH19,GTH23} unified these strands into the Perron--Frobenius theorem for the \((\boldsymbol{\sigma},\mathbf{p})\)-eigenvalue problem of nonnegative tensors,
	which either implies earlier results or improves them by requiring weaker assumptions.
	
	Alongside this, the Collatz–Wielandt formula, which provides a variational characterization of the spectral radius of nonnegative matrices, has also been extended to tensors \cite{CQZ10,CPZ08,FGH13,GTH19,LQ13,YY11}. 
	For the \((\boldsymbol{\sigma},\mathbf{p})\)-spectral radius,
	Gautier et al.\ \cite{GTH19} proposed a Collatz--Wielandt formula in terms of a multihomogeneous mapping under some sphere constraints (see \cite[Theorem~3.2(ii)]{GTH19}, stated as Lemma~\ref{pf-sp}(ii) below). 
	However, it does not align with earlier extensions of the Collatz–Wielandt formula to tensor eigenvalue problems 
	in \cite{CQZ10,CPZ08,FGH13,LQ13}. 
	In particular, the sphere constraints in Lemma~\ref{pf-sp}(ii) appear difficult to relax 
	in a straightforward manner.
	
	This difficulty prevents a direct reformulation 
	of the \((\boldsymbol{\sigma},{\bf p})\)-spectral problem as a convex optimization problem.
	This stands in contrast to the observation of Yang and Yang \cite{YY11}, who illustrated that, for a square (or rectangular) 
	nonnegative irreducible tensor, the problem of finding the spectral radius (or largest singular value) 
	can be reformulated as a convex optimization problem via the Collatz--Wielandt formula. 
	These observations naturally lead to the following question:
	
	\emph{Is there a more intrinsic Collatz--Wielandt formula for the 
		\((\boldsymbol{\sigma},{\bf p})\)-spectral radius that both unifies earlier cases 
		and admits relaxed constraints,
		so that	the general \((\boldsymbol{\sigma},{\bf p})\)-spectral problem can be reformulated as a convex optimization problem?}
	
	\subsection{Algorithmic Landscape}
	For computing the largest eigenvalue of nonnegative tensors, a variety of iterative methods have been developed, including both first-order and higher-order methods. 
	Among the first-order methods, the NQZ method \cite{NQZ10} was proposed for computing the largest  H-eigenvalue of nonnegative tensors, 
	with convergence results established for structured nonnegative tensors in \cite{CPZ11,LZI10,ZQ12}. 
	Subsequent work introduced power methods for computing the largest $\ell^{p_1,p_2}$-singular value \cite{CQYY13,LQ13}
	and $\ell^{p_1,\ldots, p_d}$-singular value \cite{FGH13} of nonnegative tensors.
	More recently, Gautier et al. \cite{GTH19,GTH23} proposed a general power method for computing the $(\bssig, {\bf p})$-spectral radius of nonnegative tensors, 
	along with a variant power method that converges under weaker assumptions.
	
	Existing higher-order methods, however, are primarily designed for computing the largest H-eigenvalue of nonnegative tensors. 
	For instance, Ni and Qi \cite{NQ15} proposed a Newton method based on equivalent nonlinear equations and further globalized it using a line search technique. 
	Sheng et al.\ \cite{SNY19} introduced an inverse iteration method with local quadratic convergence,
    while Yang and Ni \cite{YN18} developed a Chebyshev-based method with local cubic convergence.
	However, it remains unclear whether the globalization of these methods \cite{NQ15,SNY19,YN18},
	based on a positive-preserving line search procedure,
	eventually attain a unit step length.
	Consequently, the global quadratic or cubic convergence of these methods has yet to be established.
	
	Besides, the Noda iteration \cite{Noda71}, 
	originally developed for computing the largest eigenvalue of a nonnegative irreducible matrix, 
	has been extended to compute the largest H-eigenvalue of a weakly irreducible nonnegative tensor \cite{DW16,Liu22}.
	Although the higher-order Noda iteration in \cite{DW16,Liu22}
	enjoys superlinear convergence, 
	each iteration is costly, 
	as it involves solving nonsingular ${\cal M}$-tensor equations. 
	
	In parallel, combining the idea of Newton method and Noda iteration,
	Liu et al.\ \cite{LGL16}
	proposed a Newton–Noda iteration (NNI) method for computing the largest H-eigenvalue of an irreducible nonnegative third-order tensor.
	They later extended this approach to weakly irreducible nonnegative tensors with arbitrary order \cite{LGL17}. 
	A notable advantage of the NNI method is that the associated Newton equation is globally well-defined and enjoys
	a global quadratic convergence, under the assumption that the positive parameter $\{\theta_{k}\}$, determined by 
	the halving procedure in \cite{LGL17}, is bounded below by a positive constant.
	This technical assumption was recently removed by Guo \cite{Guo24}, thereby strengthening the convergence guarantee.
	
	Despite significant progress in computing the largest H-eigenvalue of nonnegative tensors, to the best of our knowledge, 
	no higher-order algorithm has yet been proposed for computing the $(\bssig, {\bf p})$-spectral radius,
	or even for computing the singular value of tensors.
	Motivated by the progress of NNI \cite{Guo24, LGL16,LGL17}, we naturally ask:
	
	\emph{Can we develop a Newton-Noda method for computing the $(\boldsymbol{\sigma}, {\bf p})$-spectral radius of nonnegative tensors with global and quadratic convergence?}
	
	\subsection{Contributions and organization}
	This paper addresses the questions posed in the preceding sections.
	
	Our first main contribution is new min-max Collatz--Wielandt formulas for the $(\boldsymbol{\sigma},\mathbf{p})$-spectral radius of nonnegative tensors
	(see Theorem~\ref{thr-minmax}). They avoid reliance on the auxiliary multihomogeneous mapping used in Lemma~\ref{pf-sp},
	recover several well-known special cases previously studied in \cite{CQZ10,CPZ08,FGH13,LQ13},
	and admit convex reformulations.
	These formulas arise naturally from a novel approach that connects the $(\boldsymbol{\sigma},\mathbf{p})$-spectral problem for nonnegative tensors to a class of convex optimization problems; see Theorems~\ref{th-str-nonneg} and~\ref{thr--weak-irr}.	
		
	Our second contribution is the development and analysis of a line search Newton-Noda method (LS-NNM) for computing the positive $(\boldsymbol{\sigma},\mathbf{p})$-eigenpair of nonnegative tensors.
	The method is formulated from both the nonlinear system and constrained optimization perspectives.
	Specifically, we derive the Newton equation from an equivalent nonlinear system and update the eigenvalue sequence
	using the strategy of the Noda iteration and its variants \cite{LGL16,LGL17,Noda71}.
	This update rule ensures that the Newton equations remain well-defined throughout the iteration.
	Under a suitable retraction, the resulting Newton direction is shown to be a feasible descent direction
	for an equivalent constrained optimization problem.
	Building on this observation,
	LS-NNM incorporates a positivity-preserving line search procedure to ensure global convergence.
	Finally, we establish global convergence and quadratic rate of LS-NNM, for the class of $(\bssig, {\bf p})$-spectral problem that admits 
	a unique positive $(\bssig, {\bf p})$-eigenpair,
	as characterized by the Perron-Frobenius theorem.
	
	The remainder of this paper is organized as follows. Section~\ref{pre} introduces notation and preliminary materials. In section~\ref{sect-prop-eig}, we present new min–max Collatz–Wielandt formulas and establish the connection between the $({\boldsymbol \sigma}, {\bf p})$-spectral problem and a class of convex optimization problems, 
	with auxiliary results provided in section~\ref{subsec-minmax-pre} and main developments in sections~\ref{subsec-minmax-str-nonneg} and \ref{subsec-minmax-weak-irr}. 
	Our algorithm is described in section~\ref{sec-LSNNM}. Its well-definedness and global convergence are studied in section~\ref{sec-gc}, and its quadratic convergence is established in section~\ref{sec-qc}. Section \ref{sec-ne} displays numerical performance of LS-NNM.
	
	\section{Notation and Preliminaries}\label{pre}
	Throughout the paper, let $\mathbb{R}$ be the set of real numbers, and $\mathbb{R}^n$ be the $n$-dimensional Euclidean space, equipped with the standard inner product $\langle \cdot, \cdot \rangle$.
	The $\ell_{p}$-norm ($p\ge 1$) on $\R^{n}$ is denoted by $\|\cdot\|_{p}$, that is, $\|{\bf x}\|_{p} \coloneq \left(\sum_{i=1}^{n}|x_{i}|^{p}\right)^{1/p}$ for any ${\bf x}=(x_1,\ldots,x_n)^\top \in \R^{n}$. Let $\mathbb{R}^{n}_{+}$ be the nonnegative orthant on $\R^{n}$, and $\Rpp^{n} \coloneq {\mathrm{int}}(\Rp^{n})$, where $\mathrm{int}(C)$ denotes the interior of a set $C\subseteq \R^{n}$.
	For ${\bf x} \in \R^{n}$, ${\bf x}\ge 0$ (resp. ${\bf x} > 0$) denotes ${\bf x} \in \Rp^{n}$ (resp. ${\bf x} \in \Rpp^{n}$). For a vector ${\bf x}\in \mathbb{R}^n_{++}$ and a scalar $\alpha\in \mathbb{R}$, we define ${\bf x}^{[\alpha]} \coloneq (x_{1}^\alpha, \ldots, x_{n}^\alpha)^{\top} \in \mathbb{R}^{n}$. We use `$\circ$' to denote the Hadamard (element-wise) product of two vectors, that is, ${\bf x} \circ {\bf y} \coloneq (x_1 y_1, \ldots, x_n y_n)^\top \in \mathbb{R}^{n}$ for any ${\bf x}=(x_1,\ldots,x_n)^\top, {\bf y}=(y_1,\ldots,y_n)^\top\in \mathbb{R}^{n}$. Then, we can further define $\frac{{\bf x}}{\bf y} \coloneq {\bf x}\circ ({\bf y}^{[-1]})$ for ${\bf x} \in \R^{n}$ and ${\bf y}\in \Rpp^{n}$.

	A real tensor ${\cal A}=(a_{j_1\cdots  j_m})$ of $m$th order and dimensions $(N_1,\ldots, N_m)$ is a multi-array with entries $a_{j_1\cdots  j_m}\in\mathbb{R}$, where $j_i\in [N_i]$ for all $i\in [m]$. Here $[n] \coloneq \{1,\ldots,n\}$ for any positive integer $n$.
	We denote the set of such tensors by $\mathbb{R}^{N_1\times\cdots \times N_m}$, and the subset of entrywise nonnegative tensors by $\mathbb{R}_{+}^{N_1 \times \cdots \times N_m}$. The set of $n\times n$ real matrices is denoted by $\mathbb{R}^{n\times n}$, and we use $I$ to denote the identity matrix in $\mathbb{R}^{n\times n}$.
	
	For a tensor ${\cal A}=(a_{j_1\cdots  j_m})\in \mathbb{R}^{N_1\times\cdots \times N_m}$, the associated multilinear form $f_{{\cal A}}:\mathbb{R}^{N_1}\times\cdots\times\mathbb{R}^{N_m}\rightarrow\mathbb{R}$ is given by
	\begin{equation}
		\label{eq-f-A-def}
		f_{{\cal A}}({\bf z}_1,\ldots,{\bf z}_m) \coloneq \sum_{j_1\in[N_1],\ldots,j_m\in[N_m]}a_{j_1\cdots j_m}z_{1,j_1}z_{2,j_2}\cdots z_{m,j_m},
	\end{equation}
	for every ${\bf z}=({\bf z}_1,\ldots,{\bf z}_m)\in \mathbb{R}^{N_1}\times\cdots\times\mathbb{R}^{N_m}$.
	The gradient of $f_{{\cal A}}$ is a mapping ${\mathscr A} = ({\mathscr A}_{1}, \ldots, {\mathscr A}_{m})$, 
	where for each $i\in [m]$, ${\mathscr A}_i=({\mathscr A}_{i,1},\ldots,{\mathscr A}_{i,N_i})$, 
	and for each $j_i\in [N_i]$, 
	the component ${\mathscr A}_{i,j_i}:\mathbb{R}^{N_1}\times\cdots\times\mathbb{R}^{N_m}\rightarrow\mathbb{R}$ is defined as
	\begin{equation}\label{eq-A-scr-def}
		{\mathscr A}_{i,j_i}({\bf z}_1,\ldots,{\bf z}_m) 
		\coloneq \sum_{
			\begin{subarray}{c}
				j_{l} \in [N_{l}], l\in [m]\setminus\{i\}
			\end{subarray}
		}	
		a_{j_1\cdots j_m}z_{1,j_1}\cdots z_{i-1,j_{i-1}} z_{i+1,j_{i+1}}\cdots z_{m,j_m}.
	\end{equation}

	\subsection{Shape partition and multihomogeneous mapping}
	Following \cite[Definition~2.1]{GTH19}, a partition ${\boldsymbol \sigma} = \{\sigma_i\}^d_{i=1}$ of $[m]$, i.e., $\cup_{i=1}^d \sigma_i=[m]$ and $\sigma_i \cap \sigma_j=\emptyset$ for any $i \neq j$, is said to be a shape partition of a tensor ${\cal A}\in \mathbb{R}^{N_1\times\cdots \times N_m}$ if it holds that: (i) $N_j=N_{j^{\prime}}$ for all $j, j^{\prime} \in \sigma_i$ and all $i \in[d]$;
	(ii) $j \leq k$ for all $j \in \sigma_i,\, k \in \sigma_{i+1}$ and $i \in[d-1]$ (if $d>1$); (iii) $|\sigma_{i}| \le |\sigma_{i+1}|$ for all $i\in [d-1]$ (if $d>1$). 
	
	The shape partition $\bssig$ of ${\mathcal A}$ can be characterized by three vector indicators ${\boldsymbol \nu}, {\bf s}, {\bf n} \in \mathbb{R}^{d}$, whose elements are given by:
	\begin{equation}\label{eq-nu-s-n-def}
		\nu_i:=|\sigma_i|,\quad s_i:=\min\{j: j\in\sigma_i\},\quad n_i:=N_{s_i}, \quad \forall i\in [d].
	\end{equation}
	Then, the size $N_1\times \cdots \times N_{m}$ of $\mathcal{A}$ can be rewritten as
	\newlength{\mylena}
	\newlength{\mylenb}
	\newlength{\mylend}
	\settowidth{\mylena}{$N_{s_1} \times \cdots \times N_{s_2-1}$}
	\settowidth{\mylenb}{$N_{s_2} \times \cdots \times N_{s_3-1}$}
	\settowidth{\mylend}{$N_{s_d} \times \cdots \times N_{m}$}
	\begin{equation*}
		\begin{array}{cccccccc}
			&\ \overbrace{N_{s_1}  \times \cdots \times  N_{s_2-1}}^{\sigma_1}  
			& \times 
			&  \overbrace{N_{s_2}  \times \cdots \times  N_{s_3-1}}^{\sigma_2}  
			& \times & \cdots & \times  
			& \overbrace{N_{s_d}  \times \cdots \times N_{m}}^{\sigma_d}
			\\& \makebox[\mylen][c]{$\rotatebox{90}{=}\ \, ~~~~~~~~~~~\, \ \rotatebox{90}{=}\ \ $}
			&   
			& \makebox[\mylen][c]{$\rotatebox{90}{=}\ \, ~~~~~~~~~~~\, \ \ \rotatebox{90}{=}\ \ $}
			&   &  &  
			& \makebox[\mylen][c]{$\ \, \rotatebox{90}{=} \ \, ~~~~~~~~~~~ \, \,
				\rotatebox{90}{=}\ \ $}
			\\& \underbrace{\makebox[\mylena][c]{$n_1\, \times \cdots \times\, n_1\ \ $}}_{\nu_1 \text { times }} 
			&  \times 
			& \underbrace{\makebox[\mylenb][c]{$n_2\, \times \cdots \times\, \ n_2\ \ $}}_{\nu_2 \text { times }}
			& \times & \cdots & \times 
			& \underbrace{\makebox[\mylend][c]{$n_d\, \times \cdots \times\, n_d$}}_{\nu_d \text { times }}
		\end{array}.
	\end{equation*}
	Let ${\bf  p}=(p_1,\ldots,p_d)\in(1,\infty)^d$, throughout the paper, we denote 
	\begin{align*}
		& \mathbb{R}^{{\boldsymbol \sigma}} \coloneq \mathbb{R}^{n_1}\times\cdots\times\mathbb{R}^{n_d}, \quad {\mathbb R}^{{\boldsymbol \sigma}}_+ \coloneq \mathbb{R}^{n_1}_+\times\cdots\times\mathbb{R}^{n_d}_+,
		\quad {\mathbb R}^{{\boldsymbol \sigma}}_{++} = \mathrm{int} ({\mathbb R}^{{\boldsymbol \sigma}}_{+}),
		\\ & {\mathcal S}^{({\boldsymbol \sigma},{\bf  p})}_+ \coloneq \{{\bf x}\in {\mathbb R}_{+}^{{\boldsymbol \sigma}} : \|{\bf x}_{i}\|_{p_i} = 1,\ i\in [d],\ x_i\in {\mathbb R}^{n_i}\}
		\quad\text{and}\quad
		{\mathcal S}^{({\boldsymbol \sigma},{\bf  p})}_{++} = {\mathcal S}^{({\boldsymbol \sigma},{\bf  p})}_+ \cap {\mathbb R}^{{\boldsymbol \sigma}}_{++}. \label{eq-sphere}
	\end{align*}
	For any ${\bf x} = ({\bf x}_1, \ldots, {\bf x}_d) \in \mathbb{R}^{{\boldsymbol \sigma}}$, ${\bf x}^{[{\boldsymbol \sigma}]}\in \mathbb{R}^{N_1}\times\cdots\times\mathbb{R}^{N_m}$ is defined by
	\begin{equation}\label{eq-x-sigma-def}
		{\bf x}^{[{\boldsymbol \sigma}]} \coloneq (\underbrace{{\bf x}_1,\ldots,{\bf x}_1}_{\nu_1~~\text{times}},\underbrace{{\bf x}_2,\ldots,{\bf x}_2}_{\nu_2~~\text{times}},\ldots,\underbrace{{\bf x}_d,\ldots,{\bf x}_d}_{\nu_d~~\text{times}}).
	\end{equation}
	
	Next, we recall the definition of multihomogeneous mapping (see, e.g., \cite[Section 2.1]{GTH19Mult}). A mapping ${P}:\mathbb{R}_+^{{\boldsymbol \sigma}} \rightarrow \mathbb{R}_+^{{\boldsymbol \sigma}}$ is said to be multihomogeneous with a nonnegative homogeneity matrix $B=(b_{ij}) \in \mathbb{R}^{d\times d}_{+}$ if it holds that
	$$
	{P}({\boldsymbol \theta} \otimes {\bf x})={\boldsymbol \theta}^B \otimes {P}({\bf x}), \quad\forall {\bf x} \in \mathbb{R}_+^{{\boldsymbol \sigma}}, {\boldsymbol \theta} \in \mathbb{R}_{+}^d,
	$$
	where ${\boldsymbol \theta} \otimes {\bf x} \coloneq \left(\theta_1 {\bf x}_1, \ldots, \theta_d {\bf x}_d\right)\in{\mathbb R}^{{\boldsymbol \sigma}}$, and ${\boldsymbol \theta}^B\in\mathbb{R}^d_+$ with entries $({\boldsymbol \theta}^B)_i \coloneq \prod_{j=1}^d \theta_j^{b_{i j}}$ for any $i\in[d]$.
	For a differentiable mapping ${F}:\mathbb{R}^{{\boldsymbol \sigma}} \rightarrow \mathbb{R}^{{\boldsymbol \sigma}}$, its Jacobian at ${\bf x} \in \mathbb{R}^{{\boldsymbol \sigma}}$ is the linear map $D{F}({\bf x})$ defined by
	$$
\textstyle	D{F}({\bf x}){\bf d}:= \lim_{t\rightarrow 0} \frac{{F}({\bf x} + t {\bf d}) - {F}({\bf x})}{t}, \quad 
	\forall {\bf d} \in \mathbb{R}^{{\boldsymbol \sigma}}.
	$$
	
	The following lemma gives a fundamental property of a differentiable multihomogeneous mapping. 
	
	\begin{lemma}\label{lm-multihomo}
		Let ${P}:\Rsigpp \to \Rsigpp$ be a differentiable multihomogeneous mapping with a homogeneity matrix $B=(b_{ij})\in \mathbb{R}^{d\times d}_{+}$. Then, it holds that
		$$
		D{P}({\bf x}) (\boldsymbol \theta \otimes {\bf x}) = (B \boldsymbol\theta)\otimes {P}({\bf x}),\quad \forall {\bf x} \in \Rsigpp,~\boldsymbol{\theta} \in \mathbb{R}_{++}^d.
		$$
	\end{lemma}
	\begin{proof}
		Let ${\bf x} \in \Rsigpp$ and $\boldsymbol{\theta} \in \mathbb{R}_{++}^d$.
		By Euler's theorem, we can obtain that
		\begin{equation*}
			\nabla_{{{\bf x}_j}} P_{i}({\bf x})^\top (\theta_{j}{\bf x}_{j}) = b_{ij} \theta_j P_{i}({\bf x}), \quad \forall i,j\in [d].
		\end{equation*}
		Then, it follows that
		\begin{equation*}
			\textstyle \nabla_{{\bf x}} P_{i}({\bf x})^\top (\theta \otimes {\bf x})
			= \sum_{j=1}^{d} \nabla_{{{\bf x}_j}} P_{i}({\bf x})^\top (\theta_{j}{\bf x}_{j}) = \sum_{j=1}^{d} b_{ij} \theta_j P_{i}({\bf x}), \quad \forall i \in [d],
		\end{equation*}
		which completes the proof.
	\end{proof}
	
	\begin{remark}\label{rm-multihom}
		The mapping ${\bf x} \mapsto {\mathscr A}({\bf x}^{[\bssig]})$ is multihomogeneous with homogeneity matrix ${\bf 1}{\boldsymbol \nu}^{\top}- I$.
	\end{remark}

	\subsection{$({\boldsymbol \sigma}, {\bf  p})$-eigenvalue problem}
	In this subsection, we recall the definitions of $\sigp$-eigenpair, two important classes of structured tensors, and the related Perron-Frobenius theorem introduced by Gautier et al. \cite{GTH19,GTH23}.
	
	\begin{definition}[$({\boldsymbol \sigma}, {\bf  p})$-eigenvalues and -eigenvectors {\cite[Definition~2.2]{GTH19}}]
		Let ${\boldsymbol \sigma} = \{\sigma_{i}\}_{i=1}^{d}$ be a shape partition of ${\cal A}\in\mathbb{R}^{N_1\times\cdots \times N_m}$, and ${\bf  p} = (p_1,\ldots,p_d)^{\top}\in(1,\infty)^d$.
		We call $(\lambda, {\bf x}) \in \mathbb{R}\times \mathbb{R}^{{\boldsymbol \sigma}}$ a $({\boldsymbol \sigma},{\bf  p})$-eigenpair of ${\mathcal A}$ if it satisfies
		\begin{equation}\label{speig}
			{\mathscr A}_{s_i}({\bf x}^{[{\boldsymbol \sigma}]})=\lambda \psi_{p_i}({\bf x}_i) \quad \mbox{and} \quad \|{\bf x}_i\|_{p_i}=1, \quad\forall i\in[d],
		\end{equation}
		where ${\mathscr A}$ and $s_i$'s are defined by \eqref{eq-A-scr-def} and \eqref{eq-nu-s-n-def}, respectively, and
		$$
		\psi_{p_i}({\bf x}_i)\coloneq\left(|x_{i,1}|^{p_i-1}{\rm sign}(x_{i,1}),\ldots,|x_{i,n_i}|^{p_i-1}{\rm sign}(x_{i,n_i})\right);
		$$
		here, ${\rm sign}(t)=t/|t|$ for $t\neq0$ and ${\rm sign}(0)=0$.
		Then, $\lambda$ and ${\bf x}$ are called a $({\boldsymbol \sigma},{\bf  p})$-eigenvalue and a $({\boldsymbol \sigma},{\bf  p})$-eigenvector of ${\mathcal A}$, respectively.
	\end{definition}
	
	\begin{definition}[{${\boldsymbol \sigma}$-nonnegativity and -irreducibility \cite[Definition~2.3]{GTH19}}]\label{struc-t}
		For a nonnegative tensor ${\cal A} \in \mathbb{R}_{+}^{N_1 \times \cdots \times N_m}$ and an associated shape partition ${{\boldsymbol \sigma}}=\left\{\sigma_i\right\}_{i=1}^d$, consider the matrix $M \in \mathbb{R}_{+}^{\left(n_1+\cdots+n_d\right) \times\left(n_1+\cdots+n_d\right)}$ defined as
		$$
		M_{\left(i, t_i\right),\left(j, l_j\right)}=\left.\frac{\partial}{\partial x_{j, l_j}} {\mathscr A}_{s_i, t_i}({\bf x}^{[{\boldsymbol \sigma}]})\right|_{{\bf x}=\mathbf{1}}, \quad \forall i,j \in [d], t_i\in [n_i], l_j\in [n_j],
		$$
		where $\mathbf{1}=(1, \ldots, 1)^{\top}$ is the vector of all ones. We say that ${\cal A}$ is
		\begin{itemize}
			\item ${\boldsymbol \sigma}$-strictly nonnegative if $M$ has at least one nonzero entry per row;
			\item ${\boldsymbol \sigma}$-weakly irreducible if $M$ is irreducible (see \cite[Definition~1.2]{BP94}).
		\end{itemize}
	\end{definition}
	\begin{remark}
		Every ${\boldsymbol \sigma}$-weakly irreducible nonnegative tensor is ${\boldsymbol \sigma}$-strictly nonnegative (see \cite[Theorem~3.1]{GTH19}).
	\end{remark}
		
	Let ${\boldsymbol \sigma} = \{\sigma_{i}\}_{i=1}^{d}$ be a shape partition of nonnegative tensor ${\cal A}\in\mathbb{R}^{N_1\times\cdots \times N_m}_{+}$, and ${\bf  p}=(p_1\ldots,p_d) \in(1,\infty)^d$.
	The $({\boldsymbol \sigma},{\bf  p})$-spectral radius of $\mathcal A$ is given by
	\begin{equation}\label{eq-radius-def}
		r^{({\boldsymbol \sigma},{\bf  p})}({\cal A}):=\sup\{|\lambda|: \lambda \text{ is a }({\boldsymbol \sigma},{\bf  p})\text{-eigenvalue of }{\cal A}\}.
	\end{equation}
	Let ${F}^{({\boldsymbol \sigma},{\bf  p})} = ({F}_1^{({\boldsymbol \sigma},{\bf  p})},\ldots,{F}_d^{({\boldsymbol \sigma},{\bf  p})}) :{\mathbb R}^{{\boldsymbol \sigma}}_+\rightarrow{\mathbb R}^{{\boldsymbol \sigma}}_+$ be defined by
	\begin{equation}\label{eq-Fsp-def}
		{F}_{i}^{({\boldsymbol \sigma},{\bf  p})}({\bf x})=({\mathscr A}_{s_i}({\bf x}^{[{\boldsymbol \sigma}]}))^{[p_{i}^{\prime} -1]} \in \mathbb{R}^{n_i}, \quad \forall i\in [d],\ {\bf x} \in \mathbb{R}^{\bssig}, 
	\end{equation}
	where $p_{i}^{\prime} := \frac{p_i}{p_i -1}$.
	Clearly, ${F}^{({\boldsymbol \sigma},{\bf  p})}$ is multihomogeneous with the homogeneity matrix $A({\boldsymbol \sigma},{\bf  p})\in\mathbb{R}^{d\times d}_+$ given by
	\begin{equation}\label{eq-A-def}
		A \coloneq A({\boldsymbol \sigma},{\bf  p}) \coloneq \text{diag}(p_{1}^{\prime} - 1,\ldots,p_{d}^{\prime}-1)({\bf 1}{\boldsymbol \nu}^{\top}- I),
	\end{equation}
	where we recall that ${\boldsymbol \nu} = \left(|\sigma_1|, \ldots, |\sigma_d|\right)^{\top}$.
	It is straightforward to see that the homogeneity matrix $A$ is nonnegative and irreducible. 
	Thus, one can further obtain a unique positive eigenvector ${\bf b} \in \mathbb{R}^{d}_{++}$ such that
	\begin{equation}\label{hmsp-eig}
		\textstyle A^{\top}{\bf b}=\rho(A){\bf b} \quad \text{and}\quad\sum_{i=1}^db_i=1,
	\end{equation}
	where $\rho(A)$ is the spectral radius of $A$.
	
	We now recall the unifying Perron-Frobenius theorem for the structured tensors given in Definition \ref{struc-t}.
	
	\begin{lemma}[{Perron--Frobenius theorem \cite[Theorem~3.2]{GTH19}}]\label{pf-sp}
		Let ${\boldsymbol \sigma}=\{\sigma_i\}^d_{i=1}$ be a shape partition of ${\cal A}\in\mathbb{R}^{N_1\times\cdots \times N_m}_{+}$, and ${\bf  p}\in(1,\infty)^d$. Furthermore, let $r^{({\boldsymbol \sigma},{\bf  p})}({\cal A})$, ${F}^{({\boldsymbol \sigma},{\bf  p})}$, $A$ and ${\bf b}$ be given by (\ref{eq-radius-def}), (\ref{eq-Fsp-def}), (\ref{eq-A-def}) and (\ref{hmsp-eig}), respectively. Suppose that ${\cal A}$ is ${\boldsymbol \sigma}$-strictly nonnegative and $\rho(A)\leq 1$. Then, the following properties hold:
		\begin{enumerate}[{\rm (i)}]
			\item There exists a $({\boldsymbol \sigma},{\bf  p})$-eigenpair $(\lambda, {\bf u})\in\mathbb{R}_+\times {\mathbb R}^{{\boldsymbol \sigma}}_{+}$ of ${\cal A}$ such that $\lambda=r^{({\boldsymbol \sigma},{\bf  p})}({\cal A})$ and ${\bf u}_i \neq 0$ for all $i\in [d]$. 
			\item Let $\gamma=\frac{\sum_{i=1}^d b_i p_i^\prime}{\sum_{i=1}^d b_i p_i^\prime-1}$; then $\gamma \in(1, \infty)$ and the following Collatz-Wielandt formula holds:
			\begin{equation}
				\label{maximality-1}
				\textstyle \max\limits_{{\bf y} \in \mathcal{S}_{+}^{({{\boldsymbol \sigma}}, {\bf p})}}\underline{\mathrm{cw}}\left({F}^{({{\boldsymbol \sigma}}, {\bf p})}, {\bf y}\right)
				= r^{({\boldsymbol \sigma}, {\bf  p})}({\cal A})
				= \inf\limits_{{\bf x} \in \mathcal{S}_{++}^{({{\boldsymbol \sigma}},{\bf p})}}\overline{\mathrm{cw}}\left({F}^{({{\boldsymbol \sigma}}, {\bf p})}, {\bf x}\right),
			\end{equation}
			where 
			\begin{align}
				& \textstyle \underline{\mathrm{cw}}\left({F}^{({{\boldsymbol \sigma}}, {\bf p})}, {\bf y}\right)=\prod_{i=1}^d\left(\min\limits_{j_i \in\left[n_i\right], y_{i, j_i}>0} \frac{{F}_{i, j_i}^{({{\boldsymbol \sigma}}, {\bf p})}({\bf y})}{y_{i, j_i}}\right)^{(\gamma-1) b_i}, \label{eq-cw-underline}
				\\ &\textstyle \overline{\mathrm{cw}}\left({F}^{({{\boldsymbol \sigma}}, {\bf p})}, {\bf x}\right)=\prod_{i=1}^d\left(\max\limits_{j_i \in\left[n_i\right]} \frac{{F}_{i, j_i}^{({{\boldsymbol \sigma}}, {\bf p})}({\bf x})}{x_{i, j_i}}\right)^{(\gamma-1) b_i}. \notag
			\end{align}
			\item If either $\rho(A)<1$ or ${\cal A}$ is ${\boldsymbol \sigma}$-weakly irreducible, then the $({\boldsymbol \sigma},{\bf  p})$-eigenvector ${\bf u}$ in \rm{(i)} can be chosen to be strictly positive, i.e., ${\bf u}\in {\mathbb R}^{{\boldsymbol \sigma}}_{++}$. Moreover, ${\bf u}$ is the unique positive $({\boldsymbol \sigma},{\bf  p})$-eigenvector of ${\cal A}$. 
		\end{enumerate}
	\end{lemma}
	
	We end this section with several fundamental properties of $Z$- and $M$-matrices.
	\begin{lemma}[\cite{BP94}]\label{lm-ZM}
		Let $M\in\mathbb{R}^{n\times n}$ be a $Z$-matrix. Then each of the following conditions is equivalent to the statement ``$M$ is a nonsingular $M$-matrix'':
		\begin{enumerate}[{\rm (i)}]
			\item   $M^{-1} \geq 0$;
			\item   $M {\bf v}\in \mathbb{R}^{n}_{++}$ for some ${\bf v}\in \mathbb{R}^{n}_{++}$.
		\end{enumerate}
	\end{lemma}
	
	\begin{lemma}[{\cite{BP94}}]\label{lm-irr-nonsing-M}
		Let $M\in\mathbb{R}^{n\times n}$ be an irreducible $Z$-matrix. 
		Then each of the following conditions is equivalent to the statement ``$M$ is a nonsingular $M$-matrix'':
		\begin{enumerate}[{\rm (i)}]
			\item $M^{-1}>0$.
			\item  $M {\bf v}\in \mathbb{R}^n_+\backslash \{0\}$ for some ${\bf v}\in \mathbb{R}^{n}_{++}$.
		\end{enumerate}
	\end{lemma}
	
	\begin{lemma}[\cite{BP94}]
		\label{lm-irr-singular-M}
		Let $M\in\mathbb{R}^{n\times n}$ be an irreducible $Z$-matrix. 
		Then, $M$ is a singular $M$-matrix if and only if there exists a unique vector ${\bf y}\in \mathbb{R}^{n}_{++}$ such that $M{\bf y} = 0$ and ${\bf 1}^{\top} {\bf y} = 1$.
		Moreover, if $M$ is a singular irreducible $M$-matrix and $M{\bf y}\ge 0$ for some ${\bf y} \in \mathbb{R}^{n}$ with ${\bf 1}^{\top} {\bf y} = 1$, then $M{\bf y} = 0$ and ${\bf y}>0$.
	\end{lemma}
	
	\section{Collatz--Wielandt formula and Convex reformulation}\label{sect-prop-eig}
	The extended Collatz--Wielandt formulas
	developed in \cite{CQZ10,CPZ08,FGH13,GTH19,LQ13,YY11} provide an important variational characterization of the spectral radius (or the largest singular value) of nonnegative tensors.
	A particularly important consequence is that, 
	for a square (or rectangular) 
	 nonnegative irreducible tensor, the problem of finding the spectral radius (or largest singular value) 
	 can be reformulated as a convex optimization problem, as demonstrated by Yang and Yang \cite{YY11}. 
	
	However, for the $(\boldsymbol{\sigma},\mathbf{p})$-spectral problem of nonnegative tensors, the sphere constraints of Collatz--Wielandt formula in Lemma~\ref{pf-sp}(ii) do not appear to be relaxed straightforwardly. This difficulty prevents
	a direct extension of the convex optimization framework in \cite{YY11} to the
	$(\boldsymbol{\sigma},\mathbf{p})$-spectral problem.
						
	The main contribution of this section is to 
	overcome this obstacle by
	establishing a new min-max Collatz--Wielandt formula for the $(\bssig, {\bf p})$-spectral radius of nonnegative tensors, stated in Theorem~\ref{thr-minmax} below.
	Unlike the previous works, our derivation is carried out
	from convex optimization perspective.
	In the following subsections, we provide an alternative approach showing
	that the $(\bssig, {\bf p})$-spectral problem can be reformulated into two convex optimization 
	problems, corresponding respectively to the min-max Collatz--Wielandt formulas \eqref{maximality-3} and \eqref{maximality} below.	
	The specific results are presented in Theorems~\ref{th-str-nonneg} and~\ref{thr--weak-irr}.
	\begin{theorem}[Min-max Collatz--Wielandt formulas]\label{thr-minmax}
		Let ${\boldsymbol \sigma}=\{\sigma_i\}^d_{i=1}$ be a shape partition of ${\mathcal A}\in\mathbb{R}^{N_1\times\cdots \times N_m}_{+}$ and ${\bf  p}=(p_1,\ldots,p_d)\in(1,\infty)^d$.
		Then the following statements hold:
		\begin{enumerate}[(i)]
			\item 
			If ${\mathcal A}$ is ${\boldsymbol \sigma}$-strictly nonnegative and $\sum_{i=1}^{d}|\sigma_i|/p_i < 1$, then
			\begin{equation}
				\label{maximality-3}
				r^{({\boldsymbol \sigma}, {\bf  p})}({\mathcal A})=\min_{{\bf x} \in {\mathbb R}_{++}^{{\boldsymbol \sigma}}}
				\left\{ 
				\max_{i\in[d], j_i\in [n_i]}
				\frac{{\mathscr A}_{s_i, j_i}\left({\bf x}^{[{{\boldsymbol \sigma}}]} \right)}{x_{i,j_i}^{p_i-1}}
				: \|{\bf x}_{1}\|_{p_1} \cdots \|{\bf x}_{d}\|_{p_d} \le 1
				\right\}.
			\end{equation}
			\item 
			If ${\mathcal A}$ is ${\boldsymbol \sigma}$-weakly irreducible and $\sum_{i=1}^{d}|\sigma_i|/p_i = 1$, then
			\begin{equation}
				\label{maximality}
				r^{({\boldsymbol \sigma}, {\bf  p})}({\mathcal A})=\min_{{\bf x} \in {\mathbb R}_{++}^{{\boldsymbol \sigma}}}
				\left\{
				\max_{i\in[d], j_i\in [n_i]}
				\frac{{\mathscr A}_{s_i, j_i}\left({\bf x}^{[{{\boldsymbol \sigma}}]} \right)}{x_{i,j_i}^{p_i-1}}
				\right\}.
			\end{equation}
		\end{enumerate}
	\end{theorem}
	\begin{proof}
		This follows directly from Theorems~\ref{th-str-nonneg} and \ref{thr--weak-irr}, which are presented later in sections~\ref{subsec-minmax-str-nonneg} and \ref{subsec-minmax-weak-irr}, respectively.
	\end{proof}
	We highlight several advantages of Theorem~\ref{thr-minmax} in comparison with Lemma~\ref{pf-sp}(ii):
	\begin{enumerate}
		[{\rm (i)}]
		\item Theorem~\ref{thr-minmax} 
		offers a more intuitive formulation,
		as it avoids the auxiliary constructs
		 $F^{(\bssig, {\bf p})}$,
		 $A$ and ${\bf b}$,
		 defined by \eqref{eq-Fsp-def}, \eqref{eq-A-def} and \eqref{hmsp-eig}, respectively.
		
		\item The expressions in \eqref{maximality-3} and \eqref{maximality} cover the well-known special cases for $d=1$, $d=2$, and $d=m$,
		previously studied in \cite{CPZ08,CQZ10,LQ13,FGH13}.
		
		\item The formulation in \eqref{maximality} 
		involves no normalized constraint, 
		while the one in \eqref{maximality-3}
		is weaker than that in \eqref{maximality-1}.
		Moreover, 
		both \eqref{maximality-3} and \eqref{maximality}
		admit convex optimization reformulations.
	\end{enumerate} 
	 
	 Throughout this section, let ${\boldsymbol \sigma}=\{\sigma_i\}^d_{i=1}$ be a shape partition of ${\mathcal A}\in\mathbb{R}^{N_1\times\cdots \times N_m}_{+}$ and ${\bf  p}=(p_1,\ldots,p_d)\in(1,\infty)^d$. Before deriving the convex optimization reformulations associated with \eqref{maximality-3} and \eqref{maximality}, we provide several auxiliary results in the next subsection. 
	
\subsection{Auxiliary results}\label{subsec-minmax-pre}
    In Lemma~\ref{pf-sp}, the condition $\rho(A)\leq 1$ plays a key role in ensuring
	the existence and uniqueness of a positive 
	$(\boldsymbol{\sigma},\mathbf{p})$-eigenpair of $\mathcal{A}$.
	Here, we provide an equivalent characterization of this condition that can be directly verified in practice. 	
	\begin{lemma}\label{lm-eqi-rA1}
		Let $A$ be given by (\ref{eq-A-def}). Then, $\rho(A)< 1$ (resp. $\rho(A)= 1$) if and only if $\sum_{i=1}^{d}|\sigma_i|/p_i < 1$ (resp. $\sum_{i=1}^{d}|\sigma_i|/p_i = 1$).
	\end{lemma}
			\begin{remark}
		In the familiar cases for $d=1$,
		$d=2$ and $d=m$, we can verify
		the condition $\rho(A)\le 1$ (resp. $=$)
		in Lemma~\ref{pf-sp}
		by the following equivalence:
		\begin{equation*}
			\rho(A)\le 1\ (\mbox{resp.}\ =) \quad \Leftrightarrow  \quad
			\begin{cases}
				p_1\ge m\ (\mbox{resp.}\  =) & \mbox{if}\ d=1;
				\\ \frac{|\sigma_1|}{p_1} + \frac{|\sigma_2|}{p_2} \le 1 \ (\mbox{resp.}\ =)& 
				\mbox{if}\ d=2;
				\\ \sum_{i=1}^{m} 1/p_i \le 1 \ (\mbox{resp.}\  =)
				& \mbox{if}\ d=m.
			\end{cases}
		\end{equation*}
		While the Perron-Frobenius theorems in \cite{FGH13,LQ13} assume respectively that $p_i\ge m$ for all $i\in [d]$ when $d=2$ and $d = m$, 
		which is within the range we have specified.
		\end{remark}
		\begin{proof}
			From the definition of $A$ in \eqref{eq-A-def}, we have
			$
			A 
			= \mathrm{diag}({\bf p} - {\bf 1})^{-1} ({\bf 1}{\boldsymbol \nu}^{\top} 
			- I),
			$
			where $\mathrm{diag}({\bf p} - {\bf 1})=\mathrm{diag}(p_1-1,\ldots,p_d-1)$ and ${\boldsymbol \nu} = \left(|\sigma_1|, \ldots, |\sigma_d|\right)^\top$. Since $A$ is nonnegative and irreducible, by Perron-Frobenius theorem, $\rho(A)< 1$ (resp. $\rho(A)=1$) if and only if there exists a vector ${\bf b}\in \mathbb{R}^{d}_{++}$ such that 
			$A{\bf b} < {\bf b}$ (resp. $A{\bf b} = {\bf b}$), 
			or, equivalently, 
			$$	
			\left(\boldsymbol 1 {\boldsymbol \nu}^\top - I\right) {\bf b} < \mathrm{diag}({\bf p} - {\bf 1}) {\bf b} \quad ({\mbox resp. }\, \left(\boldsymbol 1 {\boldsymbol \nu}^\top - I\right) {\bf b} = \mathrm{diag}({\bf p} - {\bf 1}) {\bf b}),
			$$
			that is ${\bf p}^{[-1]} {\boldsymbol \nu}^\top {\bf b} < {\bf b}\quad  ({\mbox resp. }\, {\bf p}^{[-1]} {\boldsymbol \nu}^\top {\bf b} = {\bf b}).$
			Since ${\bf p}^{[-1]} {\boldsymbol \nu}^\top$ is an irreducible nonnegative matrix with $\rho({\bf p}^{[-1]} {\boldsymbol \nu}^\top) = {\boldsymbol \nu}^\top {\bf p}^{[-1]}$,
			this is further equivalent to ${\boldsymbol \nu}^\top {\bf p}^{[-1]} < 1$ (resp. ${\boldsymbol \nu}^\top {\bf p}^{[-1]}=1$).
			This completes the proof.
		\end{proof}
		
		By Lemmas~\ref{pf-sp} and \ref{lm-eqi-rA1}, each of the following assumptions is sufficient to ensure the existence and uniqueness of a positive $({\boldsymbol \sigma}, {\bf p})$-eigenpair for $\mathcal{A}$.
		\begin{assumption}\label{ass-str-nonneg}
	$\mathcal{A}$ is ${\boldsymbol \sigma}$-strictly nonnegative and $\sum_{i=1}^{d}|\sigma_i|/p_i < 1$.
	\end{assumption}
	\begin{assumption}\label{ass-weak-irr}
	${\mathcal A}$ is ${\boldsymbol \sigma}$-weakly irreducible and $\sum_{i=1}^{d}|\sigma_i|/p_i = 1$.
	\end{assumption}
	
	Proposition \ref{pf-sp-simplify} below presents a simplified version of the Collatz–Wielandt formula. Particularly, it removes the dependence on $F^{(\bssig, {\bf p})}$, the homogeneity matrix $A$ and its positive eigenvector ${\bf b}$ in Lemma~\ref{pf-sp}(ii).
	However, the sphere constraints in \eqref{maximality-2}
	do not admit a straightforward relaxation, at least within the approaches we have considered.
	
	\begin{proposition}\label{pf-sp-simplify}
	Suppose that either Assumption~\ref{ass-str-nonneg} or Assumption~\ref{ass-weak-irr} holds.
	Then, we have
	\begin{equation}\label{maximality-2}
		\begin{split}
			& \textstyle \max\limits_{{\bf y} \in \mathcal{S}_{+}^{({{\boldsymbol \sigma}},
					{\bf p})}} 
			\left\{ 
			\min\limits_{i\in[d], j_i\in [n_i], y_{i, j_i}>0}
			\frac{{\mathscr A}_{s_i, j_i}\left({\bf y}^{[{{\boldsymbol \sigma}}]} \right)}{y_{i,j_i}^{p_i-1}}
			\right\}
			\\ & \textstyle = \ r^{({\boldsymbol \sigma}, {\bf  p})}({\mathcal A})  = \min\limits_{{\bf x} \in \mathcal{S}_{++}^{({{\boldsymbol \sigma}}, 
					{\bf p})}}
			\left\{
			\max\limits_{i\in[d], j_i\in [n_i]}
			\frac{{\mathscr A}_{s_i, j_i}\left({\bf x}^{[{{\boldsymbol \sigma}}]} \right)}{x_{i,j_i}^{p_i-1}}
			\right\}.
		\end{split}
	\end{equation}
	\end{proposition}
	\begin{proof}
		Let ${F}^{({\boldsymbol \sigma},{\bf  p})}$, $A$ and ${\bf b}$ be as in (\ref{eq-Fsp-def}), (\ref{eq-A-def}) and (\ref{hmsp-eig}), respectively.
		According to Lemma~\ref{pf-sp}(ii), for any ${\bf x} \in \mathcal{S}_{++}^{({{\boldsymbol \sigma}},{\bf p})}$,
		\begin{equation}\label{prof-maxi-2}
			\begin{split}
				& \textstyle r^{({{\boldsymbol \sigma}}, {\bf p})}({\cal A}) \leq \overline{\mathrm{cw}}\left({F}^{({{\boldsymbol \sigma}}, {\bf p})}, {\bf x}\right) = \prod_{i=1}^d \left( \max\limits_{j_i \in [n_i]} \frac{{\mathscr A}_{s_i, j_i}\left({\bf x}^{[{{\boldsymbol \sigma}}]} \right)}{x_{i,j_i}^{p_i-1}} \right)^{(\gamma - 1) b_i (p_i^{\prime} - 1)} 
				\\&\leq \prod_{i=1}^d \left(
				\max_{l\in[d], j_l\in [n_l]}
				\frac{{\mathscr A}_{s_l, j_l}\left({\bf x}^{[{{\boldsymbol \sigma}}]} \right)}{x_{l,j_l}^{p_l-1}}
				\right)^{(\gamma - 1) b_i (p_i^{\prime} - 1)} 
				= \max_{i\in[d], j_i\in [n_i]}
				\frac{{\mathscr A}_{s_i, j_i}\left({\bf x}^{[{{\boldsymbol \sigma}}]} \right)}{x_{i,j_i}^{p_i-1}}.
			\end{split}
		\end{equation}
		The first equality of (\ref{prof-maxi-2}) holds due to the equalities (\ref{eq-cw-underline}), (\ref{eq-Fsp-def}) and $p_i^{\prime} - 1 = \frac{1}{p_i - 1}$. The second equality of (\ref{prof-maxi-2}) is deduced from the fact that
		\begin{equation*}
			\textstyle	\sum_{i=1}^{d} (\gamma - 1) b_i (p_i^{\prime} - 1) = (\gamma - 1) \left(\sum_{i=1}^{d}  b_i p_i^{\prime} - 1\right) = 1,
		\end{equation*}	
		where the first equality is obtained from $\sum_{i=1}^{d}b_i = 1$ (see (\ref{hmsp-eig})), and the last equality is obtained from the definition of $\gamma$ (see Lemma~\ref{pf-sp}(ii)).
		
		Similarly, for any ${\bf y} \in \mathcal{S}_{+}^{({{\boldsymbol \sigma}},{\bf p})}$,
		\begin{equation*}
			\begin{split}
				& r^{({{\boldsymbol \sigma}}, {\bf p})}({\cal A}) \ge \underline{\mathrm{cw}}\left({F}^{({{\boldsymbol \sigma}}, {\bf p})}, {\bf y}\right) = \prod_{i=1}^d \left( \min_{j_i \in [n_i], y_{i,j_i}>0} \frac{{\mathscr A}_{s_i, j_i}({\bf y}^{[{{\boldsymbol \sigma}}]} )}{y_{i,j_i}^{p_i-1}} \right)^{(\gamma - 1) b_i (p_i^{\prime} - 1)} \\
				&\ge \prod_{i=1}^d \left(
				\min_{l\in[d], j_l\in [n_l], y_{l, j_l}>0}
				\frac{{\mathscr A}_{s_l, j_l}({\bf y}^{[{{\boldsymbol \sigma}}]} )}{y_{l,j_l}^{p_l-1}} 
				\right)^{(\gamma - 1) b_i (p_i^{\prime} - 1)} 
				= \min_{i\in[d], j_i\in [n_i], y_{i, j_i}>0}
				\frac{{\mathscr A}_{s_i, j_i}({\bf y}^{[{{\boldsymbol \sigma}}]} )}{y_{i,j_i}^{p_i-1}}.
			\end{split}
		\end{equation*}
		
		Finally, by Lemmas~\ref{pf-sp} and \ref{lm-eqi-rA1}, there exists a positive $({\boldsymbol \sigma},{\bf  p})$-eigenpair $({\bar \lambda}, {\bf {\bar x}})\in\mathbb{R}_{++}\times {\mathbb S}^{{\boldsymbol \sigma}}_{++}$ of ${\mathcal{ A}}$ with ${\bar \lambda} = r^{({{\boldsymbol \sigma}}, {\bf p})}({\cal A})$, which satisfies
		\[
		 \min_{i\in[d], j_i\in [n_i]}
		\frac{{\mathscr A}_{s_i, j_i}\left({\bf {\bar x}}^{[{{\boldsymbol \sigma}}]} \right)}{{\bar x}_{i,j_i}^{p_i-1}}
		= r^{({{\boldsymbol \sigma}}, {\bf p})}({\cal A}) 
		= \max_{i\in[d], j_i\in [n_i]}
		\frac{{\mathscr A}_{s_i, j_i}\left({\bf {\bar x}}^{[{{\boldsymbol \sigma}}]} \right)}{{\bar x}_{i,j_i}^{p_i-1}}.
		\]
		This completes the proof.
	\end{proof}
	
	The next proposition characterizes the structure of positive solutions to the $\sigp$-eigenvalue equation \eqref{speig} for nonnegative tensors.
	In particular, it shows that the 
	$d$ block normalization constraints $\|{\bf x}_i\|_{p_i} = 1$ for all $i\in [d]$ in \eqref{speig} is equivalent to a single constraint
	$\|{\bf x}_{1}\|_{p_1} \cdots \| {\bf x}_{d}\|_{p_d} = 1$, as inspired by Lim \cite{Lim2005Singular}.
		\begin{proposition}\label{prop-eig}
		Let $(\lambda, {\bf x}) \in \mathbb{R}_{++}
		\times {\mathbb R}_{++}^{{\boldsymbol \sigma}}$.
		Then
		\begin{enumerate}[{\rm (i)}]
			\item $(\lambda, {\bf x})$ is a solution of the system
			\begin{equation}\label{eq-eig}
				{\mathscr A}_{s_i}({\bf x}^{[{\boldsymbol \sigma}]}) = \lambda  {\bf x}_{i}^{[p_i - 1]}, \quad \forall i\in [d]
			\end{equation}
			if and only if there exists a positive $({\boldsymbol \sigma},{\bf  p})$-eigenpair $(\bar{\lambda}, {\bf {\bar x}})$ and a positive scalar $t>0$ such that 
			\begin{equation}
				\label{eq-x-xhat-t}
				{\bf x} = (t^{1/p_1} {\bf {\bar x}}_{1}, \ldots, t^{1/p_d}{\bf {\bar x}}_{d}) \quad \mbox{and} \quad \lambda = \bar {\lambda} t^{\sum_{i=1}^{d}\nu_i/p_i - 1}.
			\end{equation}	
			
			\item $({\lambda}, {\bf {x}})$ is a positive $({\boldsymbol \sigma},{\bf  p})$-eigenpair if and only if it satisfies (\ref{eq-eig}) and 
\begin{equation}\label{eq-sigle-constraint}
	\|{\bf x}_{1}\|_{p_1} \cdots \| {\bf x}_{d}\|_{p_d} = 1.
\end{equation}			
		\end{enumerate}
	\end{proposition}
	\begin{proof}
		(i) We first prove the ``only if'' part.
		Suppose that $(\lambda, {\bf x}) \in
		\mathbb{R}_{++} \times {\mathbb R}_{++}^{{\boldsymbol \sigma}}$ satisfies (\ref{eq-eig}).
		Let ${\barf x}\in \Rsigpp$ be given by
		\begin{equation}
			\label{eq-x-hat}
			{\bf {\bar x}} \coloneq (\theta_1^{-1} {\bf x}_{1}, \ldots, \theta_d^{-1} {\bf x}_{d}) \quad \mbox{with}\quad \theta_i \coloneq \|{\bf x}_{i}\|_{p_i},\quad \forall i \in [d].
		\end{equation}
		According to the definition of ${F}^{({\boldsymbol \sigma},{\bf  p})}$ in \eqref{eq-Fsp-def}, we have
		\begin{equation*}
			{F}_{i}^{({\boldsymbol \sigma},{\bf  p})}({\bf x}) 
			= ({\mathscr A}_{s_i}({\bf x}^{[{\boldsymbol \sigma}]}))^{[p_{i}^{\prime} -1]} 
			= \lambda^{p_{i}^{\prime} -1}  {\bf x}_{i}, \quad \forall i\in [d],
		\end{equation*}
		where $p_{i}^{\prime} -1 = \frac{1}{p_i - 1}$. From \cite[Lemma~5.1]{GTH19}, it follows that ${\bf {\bar x}}$ is a positive $\sigp$-eigenvector of $\mathcal A$. Let ${\bar \lambda}$ be the $\sigp$-eigenvalue associated with ${\barf x}$. Then, from (\ref{eq-f-A-def}) and (\ref{eq-A-scr-def}), it holds that 
		\begin{equation*}\label{eq-fA-lambda-bar}
			f_{{\cal A}}( {{\barf x}}^{[{{\boldsymbol \sigma}}]} ) = {{\barf x}}_{i}^\top {{\mathscr A}_{s_i}( {{\barf x}}^{[{{\boldsymbol \sigma}}]} )} = \bar \lambda \|{\barf x}_{i}\|_{p_i}^{p_i} = \bar \lambda,\quad \forall i\in [d].
		\end{equation*}
		By (\ref{eq-eig}) and the definition of $\theta_i$'s in (\ref{eq-x-hat}), we have
		\begin{equation*}
			f_{{\cal A}}( {{\bf x}}^{[{{\boldsymbol \sigma}}]} ) = {{\bf x}}_{i}^\top {{\mathscr A}_{s_i}( {{\bf x}}^{[{{\boldsymbol \sigma}}]} )} = \lambda \|{\bf x}_{i}\|_{p_i}^{p_i} = \lambda \theta_{i}^{p_i}, \quad \forall i\in [d],
		\end{equation*}	
		which shows that
		\begin{equation}\label{eq-theta-t}
			\theta_{i} = t^{1/p_i} \quad \mbox{with}\quad t \coloneq f_{{\cal A}}( {{\bf x}}^{[{{\boldsymbol \sigma}}]})/\lambda>0, \quad \forall i\in [d].
		\end{equation}	
		Furthermore, from (\ref{eq-x-hat}) and Remark~\ref{rm-multihom}, 
		\begin{align*}\textstyle
			f_{{\cal A}}( { {\bf {\bar x}}}^{[{{\boldsymbol \sigma}}]} ) 
			= f_{{\cal A}}( {{\bf x}}^{[{{\boldsymbol \sigma}}]} ) \prod_{i=1}^{d}  \theta_{i}^{-\nu_i} 
			= f_{{\cal A}}( {{\bf x}}^{[{{\boldsymbol \sigma}}]} ) \prod_{i=1}^{d} t^{-\nu_i/p_i}
			 = f_{{\cal A}}( {{\bf x}}^{[{{\boldsymbol \sigma}}]} ) t^{-\sum_{i=1}^{d}\nu_i/p_i}, \label{eq-fA-t}
		\end{align*}
		where the second equality holds due to (\ref{eq-theta-t}). By the definition of $t$ in (\ref{eq-theta-t}) and the above analysis, we have
		\begin{equation}\label{eq-lambda}
			\textstyle \lambda = f_{{\cal A}}( {{\bf x}}^{[{{\boldsymbol \sigma}}]})/t = f_{{\cal A}}( { {\bf {\bar x}}}^{[{{\boldsymbol \sigma}}]} ) t^{\sum_{i=1}^{d}\nu_i/p_i - 1}.
		\end{equation}
		Combining \eqref{eq-x-hat}, \eqref{eq-theta-t} and \eqref{eq-lambda} together, we conclude that (\ref{eq-x-xhat-t}) holds.

		To establish the ``if'' part, suppose that (\ref{eq-x-xhat-t}) holds. 
		Then, we claim that the pair $(\lambda, {\bf x})$ satisfies (\ref{eq-eig}). Indeed, from (\ref{eq-x-xhat-t}) and Remark~\ref{rm-multihom}, it follows that, for all $i\in [d]$, 
		\begin{equation*}
			\begin{split}
				{{\mathscr A}_{s_i}( {{\bf x}}^{[{{\boldsymbol \sigma}}]} )} 
				& \textstyle = t^{-1/p_i} \prod_{j=1}^{d}  t^{\nu_j/p_j}   {{\mathscr A}_{s_i}( {{\bf {\bar x}}}^{[{{\boldsymbol \sigma}}]} )} 
				\overset{\mathrm{(a)}}{=} {\bar \lambda} t^{-1/p_i}   t^{\sum_{j=1}^{d}\nu_j/p_j}  {\bf {\bar x}}_{i}^{[p_i-1]} 
				\\ & \textstyle \overset{\mathrm{(b)}}{=} {\bar \lambda} t^{\sum_{j=1}^{d}\nu_j/p_j-1}  {\bf x}_{i}^{[p_i-1]} 
				\overset{\mathrm{(c)}}{=} \lambda   {\bf x}_{i}^{[p_i-1]},
			\end{split}
		\end{equation*}
		where (a) follows from the definition of the $\sigp$-eigenvalue, while (b) and (c) follow from (\ref{eq-x-xhat-t}).
		This completes item (i).
		
		(ii) The ``only if'' part is clear. We turn to show the ``if'' part.
		Suppose that $(\lambda, {\bf x}) \in \mathbb{R}_{++} \times {\mathbb R}_{++}^{{\boldsymbol \sigma}}$ satisfies (\ref{eq-eig}) and \eqref{eq-sigle-constraint}.
		By item (i), there exists a positive $\sigp$-eigenpair (${\bf {\bar x}}, \bar{\lambda}$) and a scalar $t>0$ such that (\ref{eq-x-xhat-t}) holds.
		This implies that
		\begin{equation*}
			\textstyle 1 = \|{\bf x}_{p_1}\|\cdots \|{\bf x}_d\|_{p_d} = \prod_{i=1}^{d}t^{1/p_i} \|{\bf {\bar x}}_1\|_{p_1}\cdots \|{\bf {\bar x}}_d\|_{p_d} = t^{\sum_{i=1}^{d}1/p_i}.
		\end{equation*}
		Thus $t=1$ and ${\bf x} = {\bf \bar x}$. 
		This completes the proof.
	\end{proof}
	
	To investigate the min-max Collatz–Wielandt formula in \eqref{maximality-3} and \eqref{maximality}, here we define a mapping $\Phi=\left(\Phi_1, \ldots, \Phi_d\right): {\mathbb R}_{++}^{{{\boldsymbol \sigma}}} \rightarrow {\mathbb R}_{+}^{{{\boldsymbol \sigma}}}$ and a function \( c: \mathbb{R}^{\boldsymbol{\sigma}} \to \mathbb{R}_{+} \) as
	\begin{equation}\label{eq-Phi-def}
		\textstyle \Phi_{i}({\bf x}) \coloneq  \frac{{{\mathscr A}_{s_i}({\bf x}^{[{{\boldsymbol \sigma}}]} )}}{{\bf x}_{i}^{[p_i-1]}}, \quad \forall i \in [d]\quad \text{and}\quad c({\bf x}) \coloneq 
		\|{\bf x}_1\|_{p_1} \cdots \|{\bf x}_d\|_{p_d},
	\end{equation}
	where ${\mathscr A}$, $s_i$'s and ${\bf x}^{[\bssig]}$ are defined by (\ref{eq-A-scr-def}), (\ref{eq-nu-s-n-def}) and (\ref{eq-x-sigma-def}), respectively.
	
	\begin{remark}\label{rm-system-Phis-c}
		Proposition~\ref{prop-eig}(ii) implies that 
		$(\lambda, {\bf x}) \in \mathbb{R}_{++} 
		\times {\mathbb R}_{++}^{{\boldsymbol \sigma}}$ is a positive $(\bssig, {\bf p})$-eigenpair of $\mathcal{A}$ if and only if it satisfies the system:
		\begin{equation*}
			\textstyle\Phi({\bf x}) - \lambda {\bf 1} = 0 
			\quad \mbox{and} \quad
			c({\bf x}) = 1.
		\end{equation*}
	\end{remark}
	
	Now, we summarize some fundamental properties of the mapping $\Phi$ in \eqref{eq-Phi-def}.
	\begin{lemma}\label{lm-Phi}
		Let $\Phi$ be given in (\ref{eq-Phi-def}) and 
		${\bf x} \in \mathbb{R}^{{\boldsymbol \sigma}}_{++}$.
		$-D \Phi({\bf x})$ is a $Z$-matrix and
		\begin{equation}\label{eq-DPhi-px}
			\textstyle -D \Phi({\bf x}) ({\bf  p}^{[-1]} \otimes {\bf x}) 
			=  (1 - \sum_{i=1}^{d}|\sigma_i|/p_i)\Phi({\bf x}).
		\end{equation}
		Moreover, it also holds that:
		\begin{enumerate}[{\rm (i)}]
			\item If ${\cal A}$ is ${\boldsymbol \sigma}$-strictly nonnegative, then $\Phi({\bf x})>0$.
			\item If Assumption~\ref{ass-str-nonneg} holds,
			then $-D \Phi({\bf x})$ is a nonsingular $M$-matrix.
			\item If Assumption~\ref{ass-weak-irr} holds, then $-D \Phi({\bf x})$ is an irreducible singular $M$-matrix.
		\end{enumerate}
	\end{lemma}
	\begin{proof}
		For each $i\in [d]$ and $j_i \in [n_i]$, let	
$		G_{i,j_i}({\bf x}) \coloneq {{\mathscr A}_{s_i, j_i}\left({\bf x}^{[{{\boldsymbol \sigma}}]} \right)}$ for all ${\bf x} \in \mathbb{R}^{{\boldsymbol \sigma}}_{++}$.
From the definition of $\Phi$ in \eqref{eq-Phi-def},
a straightforward calculation yields
		\begin{equation}\label{eq-DPhi}
			D \Phi({\bf x}) = \diag({\bf x}^{\otimes [1-{\bf  p}]}) D G({\bf x}) - \diag( ({\bf  p}-1) \otimes G({\bf x}) \circ {\bf x}^{\otimes [-{\bf  p}]}  ),
		\end{equation}
		where we denote ${\bf z}^{\otimes [{\boldsymbol \theta}]} \coloneq ({\bf z}_{1}^{[\theta_1]},\ldots, {\bf z}_d^{[\theta_d]})$ for all ${\bf z} \in \mathbb{R}^{\boldsymbol \sigma}_{++}$ and  ${\boldsymbol \theta} \in \mathbb{R}^{d}$.
		Since ${\bf p}>1$, ${\bf x}>0$ and $D G({\bf x})\ge 0$, $-D \Phi({\bf x})$ is a $Z$-matrix.
		
		Furthermore, from the definition of $\Phi$ in (\ref{eq-Phi-def}), one can see that the homogeneous matrix of $\Phi$ is ${\bf 1} {\boldsymbol \nu}^\top - \diag({\bf  p})$, where ${\boldsymbol \nu} = (|\sigma_1|, \ldots, |\sigma_d|)^{\top}$.
		By this fact and Lemma~\ref{lm-multihomo}, we can obtain that
		\begin{equation*}
	\textstyle		-D \Phi({\bf x}) ({\bf  p}^{[-1]} \otimes {\bf x}) 
			=  -(({\bf 1} {\boldsymbol \nu}^\top - \diag({\bf  p})) {\bf  p}^{[-1]}) \otimes \Phi({\bf x})=
			(1 - \sum_{i=1}^{d}|\sigma_i|/p_i)\Phi({\bf x}).
		\end{equation*}
		
		Now, we turn to prove items (i)-(iii).
		\begin{enumerate}[(i)]
			\item Since ${\mathcal A}$ is ${\boldsymbol \sigma}$-strictly nonnegative, it follows from \cite[Lemma~6.5(ii)]{GTH19} and (\ref{eq-Fsp-def}) that $G({\bf x})>0$. Hence, $\Phi({\bf x})>0$ by definition.
			
			\item By \eqref{eq-DPhi-px} and the fact $\Phi({\bf x})>0$, we have $-D \Phi({\bf x}) ({\bf  p}^{[-1]} \otimes {\bf x})>0$. Considering that $-D \Phi({\bf x})$ is a $Z$-matrix, Lemma \ref{lm-ZM} implies that $-D \Phi({\bf x})$ is a nonsingular $M$-matrix.
			
			\item According to \cite[Lemma~6.4]{GTH19}, $DF^{({\boldsymbol \sigma}, {\bf p})}({\bf x})$ is irreducible, where $F^{({\boldsymbol \sigma}, {\bf p})}$ is defined in (\ref{eq-Fsp-def}), and
			$$
			DF^{({\boldsymbol \sigma}, {\bf p})}({\bf x}) = \diag(({\bf p^{\prime}} - {\bf 1}) \otimes (G({\bf x})^{\otimes [{\bf p^{\prime}} - 2 {\bf 1}]})) DG({\bf x}).
			$$
			Together with ${\bf p^{\prime}} - {\bf 1} = ({\bf p} - {\bf 1})^{[-1]}>0$ and $G({\bf x})>0$,
			this implies that $D G({\bf x})$ is irreducible. 
			Combining this with \eqref{eq-DPhi}, 
			we deduce that $-D \Phi({\bf x})$ is irreducible. 
			Moreover, from \eqref{eq-DPhi-px} and Assumption~\ref{ass-weak-irr}, we can obtain $-D\Phi({\bf x})\left({\bf p}^{[-1]} \otimes {\bf x}\right) = 0$.
			Since $-D\Phi({\bf x})$ is a $Z$-matrix, it follows from Lemma~\ref{lm-irr-singular-M} that $-D\Phi({\bf x})$ is an irreducible singular $M$-matrix.
		\end{enumerate}
	\end{proof}

	Next, by applying a log-exp transformation 
	as in \cite{YY11}, we convert the functions $\Phi$ and $c$ in \eqref{eq-Phi-def} into convex functions. Specifically, we define a mapping ${F} : {\mathbb R}^{{{\boldsymbol \sigma}}} \rightarrow {\mathbb R}^{{{\boldsymbol \sigma}}}$ and a function $g:\mathbb{R}^{\sigma} \to \mathbb{R}$ as follows:
	for all ${\bf y}\in {\mathbb R}^{{{\boldsymbol \sigma}}}$
	\begin{equation}\label{eq-F-def}
		\textstyle
		{F}({\bf y})  \coloneq  \log \Phi(e^{{\bf y}})\quad\text{and}\quad g({\bf y}) \coloneq \log c(e^{\bf y}) = \sum_{i=1}^{d} p_i^{-1} \log\left(\sum_{j_{i}=1}^{n_i}  e^{p_i y_{i,j_{i}}}\right),
	\end{equation}
	where, for every ${\bf z}\in {\mathbb R}_{++}^{{\boldsymbol \sigma}}$ and ${\bf w} \in {\mathbb R}^{{\boldsymbol \sigma}}$, 
	$$
	\log {\bf z} \coloneq  (\log z_{1,1},\ldots, \log z_{d,n_d})^\top\in {\mathbb R}^{{\boldsymbol \sigma}}
	\quad \mbox{and} \quad
	e^{{\bf w}} \coloneq  (e^{w_{1,1}}, \ldots, e^{w_{d,n_d}})^\top \in {\mathbb R}_{++}^{{\boldsymbol \sigma}}.
	$$
	
	\begin{remark}\label{rm-Fij-conv}
		Observing that for each $i\in [d]$ and $j_i\in [n_i]$, the function $\Phi_{i,j_i}$ is a posynomial (see \cite[Definition~4.2]{YY11}). This implies that ${F}_{i, j_i}$ is convex. Undoubtedly, log-sum-exp functions are convex. 
		Since $g$ is a positive linear combination of such functions, it is also convex.
	\end{remark}
	
	In the following, we present some key properties of $F$ and $g$ as in \eqref{eq-F-def}. By direct computation, we can obtain the gradient of $g$:
	\begin{equation}\label{nabla g}
		\nabla g({\bf y}) 
		= \begin{pmatrix}
			\|e^{{\bf y}_1}\|_{p_1}^{-p_1} e^{p_1 {\bf y}_1},\, \cdots,\, \|e^{{\bf y}_d}\|_{p_d}^{-p_d} e^{p_d {\bf y}_d}
		\end{pmatrix}^\top,
		\quad \forall y \in \Rsig,
	\end{equation}
	and the identity:
	\begin{equation}\label{eq-p-g}
		\textstyle ({\bf p}^{[-1]} \otimes {\bf 1})^{\top} \nabla g({\bf y}) \equiv \sum_{i=1}^{d} p_i^{-1}, \quad \forall {\bf y}\in \Rsig.
	\end{equation}
	
	\begin{lemma}\label{lm-F}
		Suppose that ${\cal A}$ is ${\boldsymbol \sigma}$-strictly nonnegative.
		Let $F$ and $g$ be defined by (\ref{eq-F-def}).
		Then, for any ${\bf y}\in \Rsig$, the following statements hold:
		\begin{enumerate}[{\rm (i)}]
			\item If Assumption~\ref{ass-str-nonneg} holds, then $-D F({\bf y})$ is a nonsingular $M$-matrix.
			\item If Assumption~\ref{ass-weak-irr} holds, then $-D F({\bf y})$ is an irreducible singular $M$-matrix.
			\item If Assumption~\ref{ass-str-nonneg} or \ref{ass-weak-irr} holds, then the following matrix is nonsingular:
			\begin{equation*}
				\begin{pmatrix}
					D {F}({\bf y})^\top & \nabla  g({\bf y}) \\
					{\bf 1}^\top  & 0
				\end{pmatrix}.
			\end{equation*}
			\item If Assumption~\ref{ass-str-nonneg} holds, then the system
			\begin{equation*}
				D {F}({\bf { y}})^\top {\bf  z} + \delta \nabla g({\bf y}) = 0,
				\  \langle {\bf 1}, {\bf  z}\rangle = 1
			\end{equation*}
			has a unique solution $({\bf z}, \delta) \in \Rsig \times \R$, which in fact is positive.
		\end{enumerate}
	\end{lemma}
	\begin{proof}
		Let ${\bf y} \in \mathbb{R}^{\boldsymbol \sigma}$ and ${\bf x} = e^{{\bf y}}$. By simple calculation, we have
		\begin{equation}\label{eq-grad-F}
			D {F}({\bf y}) 
			= \diag\left(\Phi(e^{{\bf y}})\right)^{-1} D \Phi(e^{{\bf y}})\diag(e^{{\bf y}})
			= \diag\left(\Phi({\bf x})\right)^{-1} D \Phi({\bf x})\diag({\bf x}).
		\end{equation}
		From Lemma~\ref{lm-Phi}(i), we know that $\Phi({\bf x})>0$. According to \eqref{eq-grad-F} and \eqref{eq-DPhi-px}, 
		\begin{equation}\label{eq-DF-p}
			\textstyle D {F}({\bf y}) ({\bf  p}^{[-1]} \otimes {\bf 1}) =  \diag\left(\Phi({\bf x})\right)^{-1} D \Phi({\bf x})
			({\bf  p}^{[-1]} \otimes {\bf x})
			= (\sum_{i=1}^{d}|\sigma_i| / p_i - 1){\bf 1}.
		\end{equation}
		
		We first address items (i) and (ii). 
		If Assumption~\ref{ass-str-nonneg} holds, Lemma~\ref{lm-Phi}(ii) implies that $-D\Phi({\bf x})$ is a nonsingular $M$-matrix. Combining this with equality (\ref{eq-grad-F}) yields item~(i). Similarly, under Assumption~\ref{ass-weak-irr}, item~(ii) follows directly from equality (\ref{eq-grad-F}) and Lemma~\ref{lm-Phi}(iii).
		
		Next, we turn to prove item (iii).
		Let $({\bf z}, \delta)\in \Rsig \times \mathbb{R}$ satisfying
		\begin{equation}
			\label{eq-multi-0}
			D {F}({\bf y})^\top {\bf z} + \delta \nabla  g({\bf y}) = 0
			\quad \text{and}\quad {\bf z}^\top \boldsymbol{1} = 0.
		\end{equation}
		It suffices to show that ${\bf z} = 0$ and $\delta = 0$.
		Indeed, left-multiplying both sides of the first equality in \eqref{eq-multi-0} by $({\bf  p}^{[-1]} \otimes {\bf 1})^{\top}$ and using (\ref{eq-p-g}) and (\ref{eq-DF-p}), we can obtain
		\begin{equation}\label{eq-delta}
		\textstyle	(\sum_{i=1}^{d}|\sigma_i| / p_i - 1){\bf 1}^{\top} {\bf z} + \delta \sum_{i=1}^{d}p_{i}^{-1} = 0.
		\end{equation}
		Then, $\delta = 0$, which in turn implies $D {F}({\bf y})^\top {\bf z} = 0$.
		
		To show ${\bf z} = 0$, we proceed in the following two steps:
		\begin{itemize}
			\item Under Assumption~\ref{ass-str-nonneg}: $-D {F}({\bf y})$ is a nonsingular $M$-matrix from item (i). This, together with $D {F}({\bf y})^\top {\bf z} = 0$, implies ${\bf z}=0$.
			\item Under Assumption~\ref{ass-weak-irr}: $-D {F}({\bf y})$ is an irreducible singular $M$-matrix from item (ii). Combining with $ (-D {F}({\bf y}))^\top {\bf z} = 0$, and invoking Lemma~\ref{lm-irr-nonsing-M}, we derive that ${\bf z} = \alpha {\bf {\bar z}}$ for some ${\bf {\bar z}} \in \Rsigpp$ and $\alpha \in \mathbb{R}$. From the second equality in \eqref{eq-multi-0}, it follows that $0 = {\bf z}^{\top} {\bf 1} = \alpha {\bf {\bar z}}^{\top}{\bf 1}$. Thus, $\alpha = 0$ and ${\bf z}=0$.
		\end{itemize}
		
		Last, we consider item (iv). According to item (iii), there exists a unique pair $({\bf z}, \delta) \in \Rsig \times \mathbb{R}$ such that
		\begin{equation}\label{eq-nl-1}
			D {F}({\bf { y}})^\top {\bf  z} + \delta \nabla g({\bf y}) = 0\quad\text{and}\quad
			\  {\bf 1}^{\top} {\bf  z} = 1.
		\end{equation}
		Observe that (\ref{eq-delta}) follows from the first equality in \eqref{eq-multi-0}, which remains valid here.
		Combining this with ${\bf 1}^\top {\bf z} =1$, we then obtain that 
		\begin{equation*}
		\textstyle	\delta = (1 - \sum_{i=1}^{d}|\sigma_i| / p_i)/\sum_{i=1}^{d}p_{i}^{-1}>0.
		\end{equation*}
		From item (i), it follows that $- D {F}({\bf { y}})$ is a nonsingular $M$-matrix. According to this fact and the first equality in \eqref{eq-nl-1}, we have that
		$
\textstyle		{\bf  z} = \delta \left(- D {F}({\bf { y}})^\top\right)^{-1} \nabla g({\bf y})>0,
		$
		where the positivity holds because of $\delta>0$, $\nabla g({\bf y})>0$ (see (\ref{nabla g})), and the fact that $[- DF({\bf  y}) ]^{-1} \ge 0$ (see Lemma~\ref{lm-ZM}).
		This completes the proof.
	\end{proof}

	\subsection{Convex optimization reformulation under Assumption~\ref{ass-str-nonneg}}
	\label{subsec-minmax-str-nonneg}
	In this subsection, under Assumption~\ref{ass-str-nonneg},
	we consider the constrained optimization problem:
	\begin{equation}\label{opt-x-str-nonneg}
		\begin{split}
			&\textstyle \min\limits_{{\bf x} \in {\mathbb R}_{++}^{{\boldsymbol \sigma}}} \left\{ 
			\phi({\bf x}) \coloneq \max_{i\in [d],j_i\in [n_i]} \Phi_{i, j_i}({\bf x}) 
			\right\}
			\\ &\mbox{s.t.} \quad c({\bf x})  \coloneq  \|{\bf x}_1\|_{p_1}\cdots \|{\bf x}_d\|_{p_d}\le 1,
		\end{split}
	\end{equation}
	and its reformulation:
	\begin{equation}\label{opt-conv-str-nonneg}
		\begin{split}
			& \min_{({\bf y},\, \mu) \in {\mathbb R}^{{\boldsymbol \sigma}} \times \mathbb{R}} \mu  \quad \mbox{s.t.} \quad   {F}({\bf y}) -  \mu{\bf 1} \le 0 \quad \mbox{and} \quad g({\bf y}) \le 0,
		\end{split}
	\end{equation}
	where $\Phi$, $F$ and $g$ are defined by \eqref{eq-Phi-def} and (\ref{eq-F-def}), respectively. From Remark~\ref{rm-Fij-conv}, the functions $F_{i,j_i}$'s and $g$ are convex. Therefore, problem (\ref{opt-conv-str-nonneg}) is a convex program.
	
	\begin{remark}\label{rm-equi-str-nonneg}
		If $({\bf { y}}, \mu) \in {\mathbb R}^{{\boldsymbol \sigma}} \times \mathbb{R}$ is a solution to (\ref{opt-conv-str-nonneg}), then ${\bf x} = e^{\bf y}$ is a solution to (\ref{opt-x-str-nonneg}) and $\phi({\bf x}) = e^{\mu}$.
		Conversely, if ${\bf x}$ is a solution to (\ref{opt-x-str-nonneg}), then $({\bf { y}}, \mu) = (\log {\bf x}, \log {\phi({\bf x})})$ is a solution to (\ref{opt-conv-str-nonneg}).
	\end{remark}
	
	Now, we establish the connection between the $(\bssig, {\bf p})$-spectral problem and the convex optimization problem introduced above.
	
	\begin{theorem}\label{th-str-nonneg}
        Suppose that Assumption~\ref{ass-str-nonneg} holds. Let ${\barf x} \in \Rsigpp$ be the unique positive $\sigp$-eigenvector of ${\mathcal A}$, and ${\barf y} = \log {\barf x}$. Then it holds that:
		\begin{enumerate}[{\rm (i)}]
			\item Problem (\ref{opt-x-str-nonneg}) has a unique solution ${\barf x}$ with optimal value $\phi({
				\bf \bar x}) = r^{\sigp}({\mathcal A})$;
			
			\item Problem (\ref{opt-conv-str-nonneg}) has a unique solution $({\barf y}, \bar \mu) = \left(\log \barf{x},\, \log r^{\sigp}({\mathcal A}) \right)$.
		\end{enumerate}
	\end{theorem}
	\begin{proof}
		Let the Lagrangian function of problem (\ref{opt-conv-str-nonneg}) be 
		$$
		{\tilde{L}}({\bf y}, \mu, {\bf z}, \delta) = \mu + \langle {\bf z}, {F}({\bf y}) -  \mu{\bf 1}\rangle + \delta g({\bf y}), \quad \forall ({\bf y}, \mu, {\bf z}, \delta) \in \Rsig \times \mathbb{R} \times \Rsig \times \R.
		$$
		Since Slater's condition holds for problem (\ref{opt-conv-str-nonneg}), the pair $({\bf { y}}, \mu) \in {\mathbb R}^{{\boldsymbol \sigma}} \times \mathbb{R}$ is a solution to the problem (\ref{opt-conv-str-nonneg}) if and only if there exists a Lagrange multiplier $({\bf  z}, \delta) \in {\mathbb R}_{+}^{{\boldsymbol \sigma}}\times \R_{+}$ satisfying KKT conditions:
		\begin{align}
			& D {F}({\bf { y}})^\top {\bf  z} + \delta \nabla g({\bf y}) = 0,
			\  {\bf 1}^\top {\bf  z} = 1, \label{eq-DF-z-conv}
			\\ & {F}({\bf y}) -  \mu{\bf 1}\le 0, 
			\quad 
			{\bf z}^{\top}({F}({\bf y}) -  \mu{\bf 1}) = 0, \label{eq-z-F-conv}
			\\ & g({\bf y}) \le 0 \quad \mbox{and}\quad \delta g({\bf y}) = 0. \label{eq-slack-g}
		\end{align}
		
		From Lemma~\ref{lm-F}(iv), for every ${\bf { y}} \in {\mathbb R}^{{\boldsymbol \sigma}}$, the linear system (\ref{eq-DF-z-conv}) has a unique solution $({\bf  z}_{\bf  y}, \delta_{\bf y}) \in \Rsig \times \R$, which is positive. Let $({\bf z}, \delta) = ({\bf z}_{\bf y}, \delta_{\bf y})$, the KKT conditions (\ref{eq-DF-z-conv})-(\ref{eq-slack-g}) reduce to
		\begin{equation}\label{eq-F-f-conv}
			{F}({\bf { y}}) - {\mu} {\bf 1} = 0 \quad \mbox{and} \quad g({\bf y}) = 0.
		\end{equation}
		According to the transformation ${\bf  x} = e^{\bf  y}$ and $\lambda = e^{\mu}$, and the definition of $F$ and $g$ in (\ref{eq-F-def}), the system (\ref{eq-F-f-conv}) is equivalent to 
		\begin{equation*}\label{eq-Phi-lambda-conv}
			\Phi({\bf  x}) - \lambda {\bf 1} = 0 \quad \mbox{and} \quad c(\bf x) = 1.
		\end{equation*}	
		From Remark~\ref{rm-system-Phis-c}, this is further equivalent to the fact that $(\lambda, {\bf x})$ is a positive $\sigp$-eigenpair of ${\mathcal A}$. By Lemmas~\ref{pf-sp} and \ref{lm-eqi-rA1}, ${\mathcal A}$ has a unique positive $\sigp$-eigenpair $(\bar \lambda, {\bf \bar x})\in \Rpp \times \Rsigpp$ and $\bar \lambda = r^{\sigp}({\mathcal A})$. Combining this with the equivalences above, we conclude item (ii).
		Item (i) then follows from item (ii) and Remark~\ref{rm-equi-str-nonneg}.
	\end{proof}

	\subsection{Convex optimization reformulation under Assumption~\ref{ass-weak-irr}}
	\label{subsec-minmax-weak-irr}
	In this section, under Assumption~\ref{ass-weak-irr}, we consider the optimization problem:
	\begin{equation}\label{opt-x-weak-irr}
		\begin{split}
			&\textstyle \min\limits_{{\bf x} \in {\mathbb R}_{++}^{{\boldsymbol \sigma}}} \left\{ 
			\phi({\bf x}) \coloneq \max_{i\in [d],j_i\in [n_i]} \Phi_{i, j_i}({\bf x}) 
			\right\},
		\end{split}
	\end{equation}
	and its reformulation:
	\begin{equation}\label{opt-conv-weak-irr}
		\begin{split}
			& \min_{({\bf y},\, \mu) \in {\mathbb R}^{{\boldsymbol \sigma}} \times \mathbb{R}} \mu  \quad \mbox{s.t.} \quad   {F}({\bf y}) -  \mu{\bf 1} \le 0,
		\end{split}
	\end{equation}
	where $\Phi$ and $F$ are defined by \eqref{eq-Phi-def} and (\ref{eq-F-def}), respectively.
	
	\begin{remark}\label{rm-equi-weak-irr}
		$({\bf { y}}, \mu) \in {\mathbb R}^{{\boldsymbol \sigma}} \times \mathbb{R}$ is a solution to (\ref{opt-conv-weak-irr}) if and only if ${\bf x} = e^{\bf y}$ is a solution to (\ref{opt-x-weak-irr}) and $\phi({\bf x}) = e^{\mu}$.
	\end{remark}
	
	We now establish the relationship between the \( ({\boldsymbol \sigma}, {\bf p}) \)-spectral problem and the optimization problems \eqref{opt-x-weak-irr} and \eqref{opt-conv-weak-irr}. 
	
	\begin{theorem}\label{thr--weak-irr}
        Suppose that Assumption~\ref{ass-weak-irr} holds. Let ${\barf x} \in \Rsigpp$ be the unique positive $\sigp$-eigenvector of ${\mathcal A}$, and ${\barf y} = \log {\barf x}$. Then it holds that:
		\begin{enumerate}[{\rm (i)}]
			\item The solution set ${S}^*_{1}$ of problem (\ref{opt-x-weak-irr}) is nonempty, which is given by
			\begin{equation*}\label{sol-x}
				S^*_{1} = \{(t^{1/p_1} {\barf x}_{1}, \ldots, t^{1/p_d}{\barf x}_{d})\in \Rsigpp : 
				t\in \mathbb{R}_{++}
				\},
			\end{equation*}
			and the optimal value of problem (\ref{opt-x-weak-irr}) is $\phi({
				\bf \bar x}) = r^{\sigp}({\mathcal A})$.
			
			\item The solution set $S^*_{2}$ of problem (\ref{opt-conv-weak-irr}) is nonempty, which is given by
			\begin{equation}\label{sol-y}
				S^*_{2} = \{ ({\bf {\bar y}} - \beta ({\bf  p}^{[-1]} \otimes {\bf 1}), {\bar \mu} ) \in \Rsig \times \mathbb{R}: \beta \in \mathbb{R}\},
			\end{equation}
			where ${\bar \mu} = \log r^{\sigp}({\mathcal A})$.
		\end{enumerate}
	\end{theorem}
	\begin{proof}
		Denote the Lagrangian function of problem (\ref{opt-conv-weak-irr}) as 
		$$
		L({\bf y}, \mu, {\bf z}) = \mu + \langle {\bf z}, {F}({\bf y}) -  \mu{\bf 1}\rangle, \quad \forall ({\bf y}, \mu, {\bf z}) \in \Rsig \times \mathbb{R} \times \Rsig.
		$$
		Since Slater's condition holds for problem (\ref{opt-conv-weak-irr}), the KKT conditions are necessary and sufficient for optimality. Thus, a pair $({\bf { y}}, \mu) \in {\mathbb R}^{{\boldsymbol \sigma}} \times \mathbb{R}$ solves (\ref{opt-conv-weak-irr}) if and only if there exists a Lagrange multiplier ${\bf  z} \in {\mathbb R}_{+}^{{\boldsymbol \sigma}}$ such that
		\begin{align}
			& D {F}({\bf { y}})^\top {\bf  z}  = 0,
			\quad {\bf 1}^{\top} {\bf  z} = 1, \label{eq-DF-z}
			\\ & {\bf z}^{\top}({F}({\bf y}) -  \mu{\bf 1}) = 0, \label{eq-zF}
			\\ & {\bf z}\ge 0 \quad \mbox{and} \quad {F}({\bf y}) -  \mu{\bf 1}\le 0. \label{eq-z-F}
		\end{align}
		
		By Lemma~\ref{lm-F}(ii), for any ${\bf { y}} \in {\mathbb R}^{{\boldsymbol \sigma}}$, the matrix $-D F({\bf  y})$ is an irreducible singular $M$-matrix. Applying Lemma~\ref{lm-irr-nonsing-M}, there exists a unique ${\bf z}_{\bf y}\in \Rsigpp$ satisfying \eqref{eq-DF-z}. Substituting ${\bf z} = {\bf z}_{\bf y}$ into the KKT conditions \eqref{eq-DF-z}-\eqref{eq-z-F}, we can obtain
		\begin{equation}\label{eq-F-f}
			{F}({\bf { y}}) - {\mu} {\bf 1} = 0.
		\end{equation}
		Applying the transformation ${\bf  x} = e^{\bf  y}$ and $\lambda = e^{\mu}$, and the definition of $F$ in (\ref{eq-F-def}), \eqref{eq-F-f} is equivalent to 
		\begin{equation}\label{eq-Phi-lambda}
			\Phi({\bf  x}) - \lambda {\bf 1} = 0.
		\end{equation}	
		
		From Assumption~\ref{ass-weak-irr} and Lemmas~\ref{pf-sp} and \ref{lm-eqi-rA1}, tensor ${\mathcal A}$ has a unique positive $\sigp$-eigenpair $(\bar \lambda, {\bf \bar x})$ with $\bar \lambda = r^{\sigp}({\mathcal A})$. Furthermore, Proposition~\ref{prop-eig}(i) implies that $(\lambda, {\bf x}) \in \mathbb{R}_{++} \times {\mathbb R}_{++}^{{\boldsymbol \sigma}}$ satisfies (\ref{eq-Phi-lambda}) if and only if there exists a scalar $t>0$ such that 
		\begin{equation*}
			{\bf x} = (t^{1/p_1} {\bf {\bar x}}_{1}, \ldots, t^{1/p_d}{\bf {\bar x}}_{d}) \quad \mbox{and} \quad \lambda = \bar {\lambda} t^{\sum_{i=1}^{d}|\sigma_i|/p_i - 1} = r^{\sigp}({\mathcal A}).
		\end{equation*}	
		Consequently, $({\bf { y}}, \mu)$ solves problem (\ref{opt-conv-weak-irr}) if and only if there exists a scalar $t>0$ such that
		$$
		{\bf y} = \log (t^{1/p_1} {\barf x}_{1}, \ldots, t^{1/p_d}{\barf x}_{d}) = \log {\bf {\bar x}} + (\log t) ({\bf  p}^{[-1]} \otimes {\bf 1}) 
		\quad \mbox{and} \quad \mu = \log  r^{\sigp}({\mathcal A}).
		$$
		This completes the proof of item (ii), which, combined with Remark~\ref{rm-equi-weak-irr}, also establishes item (i).
	\end{proof}

	We conclude this subsection with the following proposition, 
	which plays a pivotal role in the development of a positivity-preserving algorithm for computing the $(\bssig, {\bf p})$-spectral radius, to be considered in the next section. 
	We note that the special case $d=1$ of this result appeared earlier in \cite[Lemma~4]{LGL17} in the context of computing Perron pairs of weakly irreducible nonnegative tensors.
	
	\begin{proposition}\label{pro-level-bdd}
		Let $\phi$ and $c$ be defined in \eqref{opt-x-str-nonneg}.
		Let ${\bf x}^0 \in {\mathbb R}_{++}^{{\boldsymbol \sigma}}$ be such that $c({\bf x}^0) = 1$.
		Suppose that either Assumption~\ref{ass-str-nonneg} or~\ref{ass-weak-irr} holds.
		Then the set 
			\begin{equation*}\label{eq-level}
			W \coloneqq \{{\bf x} \in {\mathbb R}_{++}^{\boldsymbol \sigma} : \phi({\bf x}) \le \phi({\bf x}^0),\, c({\bf x}) = 1\}
		\end{equation*}
		is compact, and it is bounded away from zero in the sense that
		$$\textstyle
		\inf\limits_{{\bf x}\in W}\min\{ x_{i,j_i} : i\in [d], j_i\in [n_i]\} > 0.
		$$ 
	\end{proposition}
	\begin{proof}
		Let ${\bf y}^0 = \log {\bf x}^0$ and $\mu_0 = \log \phi({\bf x}^{0})$. Define 
		$$
		\widehat{W} \coloneqq \{({\bf y}, \mu) \in {\mathbb R}^{\boldsymbol \sigma} \times \mathbb{R}: \mu \le \mu_0, F({\bf y}) - \mu {\bf 1} \le 0,\, g({\bf y}) = 0\},
		$$
		where $F$ and $g$ are defined in \eqref{eq-F-def}. Obviously, $({\bf y}^{0}, \mu_0) \in \widehat{W}$, then $\widehat{W} \neq \emptyset$. Moreover, according to the definition of $W$ and $\widehat{W}$, ${\bf x} \in W$ if and only if there exists a scalar $\mu \in \mathbb{R}$ such that $({\bf y}, \mu) \in \widehat{W}$ with ${\bf x} = e^{\bf y}$. If $\widehat{W}$ is bounded, then $W$ is closed, bounded, and bounded away from zero. Hence, in the following, we are committed to proving that $\widehat{W}$ is bounded.

		Indeed, if Assumption~\ref{ass-str-nonneg} holds, then according to Theorem~\ref{th-str-nonneg}(ii), the convex programming (\ref{opt-conv-str-nonneg}) has a unique solution. It follows that the level set 
		$$
		\Omega \coloneqq \{({\bf y}, \mu) \in {\mathbb R}^{\boldsymbol \sigma} \times \mathbb{R}: \mu \le \mu_0, F({\bf y}) - \mu {\bf 1} \le 0,\, g({\bf y}) \le 0\}
		$$
		of \eqref{opt-conv-str-nonneg} is bounded. This implies that ${\widehat W} \subseteq \Omega$ is bounded, too.
		
		Suppose that Assumption~\ref{ass-weak-irr} holds. Consider the level set $\widehat{\Omega}$ of \eqref{opt-conv-weak-irr}:
		$$
		\widehat{\Omega} \coloneqq \{({\bf y}, \mu) \in {\mathbb R}^{\boldsymbol \sigma} \times \mathbb{R}: \mu \le \mu_0, F({\bf y}) - \mu {\bf 1} \le 0\},
		$$
		which is nonempty (since $({\bf y}^{0}, \mu_0)\in \widehat{\Omega}$), closed and convex. Let $S^{*}_{2}$ be defined in (\ref{sol-y}).
		By \cite[Theorem~8.7]{R97convex} and Theorem~\ref{thr--weak-irr}(ii), we obtain that
		\[
		O^{+} {\widehat \Omega} = O^{+}S^{*}_{2}=\{ \beta ({\bf  p}^{[-1]} \otimes {\bf 1}, 0) : \beta \in \mathbb{R}\},
		\]
		where $O^{+} {C}$ denotes the recession cone of a convex set $C\subseteq \Rsig$.
		
		Suppose on the contrary that there exists an unbounded sequence 
		$\{({\bf y}^{k}, \mu_k)\}_{k\ge 1} \subseteq {\widehat{W}} \subseteq \widehat{\Omega}$ such that 
		$({\bf y}^{k}, \mu_{k}) \neq ({\bf y}^{0}, \mu_0)$ for all  $k\ge 1$. 
		Let ${\bf u}^{k}  \coloneq  
		\frac{({\bf y}^{k} - {\bf y}^{0}, \mu_{k} - \mu_0)}
		{\|({\bf y}^{k} - {\bf y}^{0}, \mu_{k} - \mu_0)\|}$ for all  $k\ge 1$.
		Since every accumulated point of $\{{\bf u}^{k}\}$ belongs to $O^{+} {\widehat{\Omega}}$ and $\|{\bf u}^{k}\| = 1$ for all $k \ge 1$, by passing to a subsequence if necessary, we may assume that 
		\begin{equation}\label{eq-beta-def}
			{\bf u}^{k} \rightarrow \beta_0 ({\bf  p}^{[-1]} \otimes {\bf 1}, 0) \quad \mbox{as} \quad k \rightarrow \infty,
		\end{equation}	
		where $\beta_0 \in \mathbb{R}$ with $|\beta_0 |= \|{\bf  p}^{[-1]} \otimes {\bf 1}\|^{-1} \neq 0$.
		
		From (\ref{nabla g}) and \eqref{eq-p-g}, it follows that 
		\begin{equation}
			\label{eq-gk-def}
		\textstyle	{\bf g}^{k}  \coloneq  \nabla g({\bf y}^{k}) > 0
			\quad \mbox{and} \quad
			({\bf  p}^{[-1]} \otimes {\bf 1})^\top {\bf g}^{k} = \sum_{i=1}^{d}p_i^{-1}, \quad \forall k \ge 0.
		\end{equation}
		Since $\{{\bf g}^{k}\}$ is bounded, by passing to a subsequence if necessary, we may assume without loss of generality that ${\bf g}^{k} \rightarrow {\bf {\bar g}}$ as $k \rightarrow \infty$. It then follows that
		\begin{equation}\label{eq-beta-1}
			\begin{split}
				\textstyle\beta_0 \sum_{i=1}^{d}p_i^{-1} & \textstyle = \beta_0 {\bf {\bar g}}^\top ({\bf  p}^{[-1]} \otimes {\bf 1}) 
				= \lim\limits_{k \rightarrow \infty} \langle ({\bf g}^{k}, 0), 
				{\bf u}^k\rangle 
				\\ & \textstyle = \lim\limits_{k \rightarrow \infty} \frac{({\bf g}^k)^\top ({\bf y}^{k} - {\bf y}^0)}{\|({\bf y}^{k} - {\bf y}^{0}, \mu_{k} - \mu_0)\|} \ge 0,
			\end{split}
		\end{equation}
		where the last inequality holds because of $({\bf g}^k)^\top ({\bf y}^{k} - {\bf y}^0) \ge g({\bf y}^0) - g({\bf y}^k) = 0$
		($g$ is convex and $g({\bf y}^{k}) = 0$ for all $k\ge 0$). On the other hand, by \eqref{eq-gk-def}, we obtain
		\begin{equation}\label{eq-beta-2}
			\begin{split}
				\textstyle\beta_0 \sum_{i=1}^{d}p_i^{-1} &\textstyle = \beta_0 ({\bf g}^{0})^\top ({\bf  p}^{[-1]} \otimes {\bf 1}) 
				= \lim\limits_{k \rightarrow \infty} \langle ({\bf g}^{0}, 0), 
				{\bf u}^k\rangle 
				\\ & \textstyle = \lim\limits_{k \rightarrow \infty} \frac{({\bf g}^0)^\top ({\bf y}^{k} - {\bf y}^0)}{\|({\bf y}^{k} - {\bf y}^{0}, \mu_{k} - \mu_0)\|} \le 0,
			\end{split}
		\end{equation}
		where the last inequality holds because of $({\bf g}^0)^\top ({\bf y}^{k} - {\bf y}^0) \le g({\bf y}^k) - g({\bf y}^0) = 0$.
		
		Then, \eqref{eq-beta-1} and \eqref{eq-beta-2} together claim that $\beta_0 = 0$, which contradicts to the definition of $\beta_0$ in \eqref{eq-beta-def}. Consequently, ${\widehat{W}}$ is bounded.
	\end{proof}
	
	\section{A line search Newton-Noda method}\label{sec-LSNNM}
We develop a line search Newton–Noda method (LS-NNM)
for computing the positive $(\boldsymbol{\sigma},{\bf p})$-eigenpair of a nonnegative tensor.
It constructs the Newton equation following the idea of the Newton-Noda iteration \cite{LGL16,LGL17}.
Specificly, in contrast to the halving process used in \cite{LGL17},
we introduce a positivity-preserving descent line search
from optimization perspective to ensure globalization of the method.

	Throughout this section, we assume that ${\bf p}=(p_1,\ldots ,p_d)\in(1,\infty)^d$, and ${\boldsymbol \sigma}=\{\sigma_i\}^d_{i=1}$ is a shape partition of a tensor ${\mathcal A}\in\mathbb{R}^{N_1\times\cdots \times N_m}_{+}$. 
	Observe from Remark~\ref{rm-system-Phis-c} that a pair $({\bf x}, \lambda) \in {\mathbb R}_{++}^{{\boldsymbol \sigma}}\times \mathbb{R}_{++}$ is a positive $(\bssig, {\bf p})$-eigenpair of $\mathcal{A}$ 
	if and only if it satisfies the nonlinear system:
	\begin{equation}\label{eq-def-H}
		H({\bf x}, \lambda)  \coloneq  \begin{pmatrix}
			r({\bf x}, \lambda)\\
			c({\bf x})-1
		\end{pmatrix} = 0,
	\end{equation}
	where
	\begin{equation}\label{eq-r-def}
		r({\bf x}, \lambda) \coloneqq -\Phi({\bf x})\circ {\bf x} + \lambda {\bf x},
		\quad \forall {\bf x} \in {\mathbb R}_{++}^{{\boldsymbol \sigma}},
	\end{equation}
	and here $\Phi$ and $c$ are defined by (\ref{eq-Phi-def}). 
	
	Besides, by Theorem~\ref{thr-minmax}
	and Proposition~\ref{prop-eig}(ii), the positive $(\bssig, {\bf p})$-eigenvector of $\mathcal{ A}$ also solves the constrained optimization problem:
	\begin{equation}
		\begin{split}
			\label{opt-phi-c=1}
			& \min_{{\bf x}\in \Rsig}  
			\left\{ \phi({\bf x})  \coloneq \max_{i\in [d],j_i\in [n_i]} \Phi_{i, j_i}({\bf x})
			\right\}
			\\ & \mbox{s.t.} \quad 
			{\bf x} \in C\coloneqq \{ {\bf x}\in \Rsigpp :  c({\bf x}) = 1\}.
		\end{split}
	\end{equation}
	To enforce feasibility, we introduce a retraction mapping 
	${R}:\{{\bf x}\in {\mathbb R}^{{\boldsymbol \sigma}}: 
	c({\bf x}) \neq 0\} \rightarrow \Rsig$
	defined by
	\begin{equation}
		\label{eq-R-def}
		{R}({\bf x}) = \frac{{\bf x}}{\sqrt[d]{c({\bf x})}},
		\quad \forall {\bf x} \in \{{\bf x}\in {\mathbb R}^{{\boldsymbol \sigma}}: 
		c({\bf x}) \neq 0\}.
	\end{equation}
	It is straightforward to see that $R({\bf x}) \in C$ for all ${\bf x} \in \Rsigpp$, and that
	\begin{equation}\label{eq-DR-def}
		DR({\bf x}) = I - \tfrac{1}{d} {\bf x} \nabla c({\bf x})^{\top}, \quad \forall {\bf x} \in C.
	\end{equation}	
	
	Now, we present our algorithm, referred to as LS-NNM, for computing the positive $(\boldsymbol{\sigma},{\bf p})$-eigenpair of $\mathcal{A}$
	(or equivalently, solving \eqref{eq-def-H} and \eqref{opt-phi-c=1}) in Algorithm~\ref{alg-H-y} below.	
	
			\begin{algorithm}[H]
		\caption{A line search Newton-Noda method (LS-NNM) for \eqref{eq-def-H} and \eqref{opt-phi-c=1}.}
		\label{alg-H-y}
		\begin{algorithmic}[1]
			\setlength{\columnwidth}{\linewidth}
			\STATE Given $\sigma\in (0,1)$, $ \rho \in (0, 1)$ and ${\bf u}^0\in C$.
			Let ${\bf x}^{0} = R({\bf u}^{0})$, where $R$ is defined in \eqref{eq-R-def}.
			Let $\lambda_0= \phi({\bf x}^0)$ and $k:=0$.
			\WHILE{$H({\bf x}^k, \lambda_k)\neq 0$}
			\STATE Compute $({\bf d}^k, \delta_k)$ by solving the Newton equation associated with \eqref{eq-def-H}:
			\begin{equation}\label{eq-nt}
				DH({\bf x}^{k}, \lambda_{k})
				\begin{pmatrix}
					{\bf d}^{k}\\
					\delta_{k}
				\end{pmatrix} + 
				H({\bf x}^{k}, \lambda_{k}) = 0.
			\end{equation}
			\STATE 	Find the largest $\alpha_k\in \{\rho^{i}: i=0,1,\ldots\}$ 
			satisfying:
			\begin{equation}\label{ls}
				{\bf x}^k + \alpha_k {\bf d}^k > 0
				\quad \mbox{and} \quad
				\phi( R({\bf x}^k + \alpha_k {\bf d}^k) ) 
				\le \phi({\bf x}^{k}) + \sigma \alpha_k \delta_{k}.
			\end{equation}
			\vspace{-.45cm}
			\STATE Let ${\bf x}^{k+1}=R({\bf x}^k + \alpha_k {\bf d}^k)$ 
			and $\lambda_{k+1}= \phi({\bf x}^{k+1})$. 
			Let $k\leftarrow k+1$.
			\ENDWHILE
		\end{algorithmic}
	\end{algorithm}
	
	In essence, Algorithm~\ref{alg-H-y} is guided by two key principles. 
	First, the update of the $(\bssig, {\bf p})$-eigenvalue
	sequence $\{\lambda_k\}$ follows the idea of NNI method \cite{Noda71,LGL16,LGL17}.
	Specifically, we set $\lambda_{k+1} = \phi({\bf x}^{k+1})$ at each iteration, 
	in contrast to classical Newton method, which generally updates the eigenvalue by $\tilde{\lambda}_{k+1} = \lambda_k + \alpha_k \delta_k$
	with a suitable stepsize $\alpha_k$.
	This update rule ensures that, 
	at each iteration,
	the Jacobian $DH({\bf x}^{k}, \lambda_{k})$ is nonsingular 
	(see Lemma~\ref{nonsingular} below), and hence the Newton direction in \eqref{eq-nt} is always well-defined.
	
	Second, if ${\bf x}^{k}\in \Rsigpp$ and $\lambda_{k} = \phi({\bf x}^{k})$ is not a $(\bssig, {\bf p})$-eigenpair of $\mathcal{A}$,
	then ${\bf d}^{k}$ given in \eqref{eq-nt}
	indeed is a descent direction
	of the composite function $\phi\circ R$ (see Lemma~\ref{glo} below). 
	This allows us to employ a
	positive-preserving Armijo line search procedure (i.e. \eqref{ls}),	
	ensuring that the iterates
	remain feasible and that the objective function in \eqref{opt-phi-c=1} is nonincreasing. 
	In this sense, Algorithm~\ref{alg-H-y} can also be
	viewed as a feasible descent method 
	for \eqref{opt-phi-c=1}.
	This new insight makes LS-NNM more implementable and transparent compared with the halving process used in \cite{LGL17}.

We are going to present the well-definedness and the global
convergence of LS-NNM in section \ref{sec-gc},
and the quadratic convergence in section~\ref{sec-qc}.

	\subsection{Well-definedness and global convergence}\label{sec-gc}
	A straightforward calculation yields the Jacobian of $H$ in \eqref{eq-def-H} as follows: 
\begin{equation}\label{eq-nabla-H}
	D H({\bf x}, \lambda)  
	=  		
	\begin{pmatrix}
		J({\bf x}, \lambda) & {\bf x}\\
		\nabla c({\bf x})^\top & 0
	\end{pmatrix}, \quad 
	\forall ({\bf x}, \lambda) \in {\mathbb R}_{++}^{{\boldsymbol \sigma}}\times \mathbb{R}_{++},
\end{equation}
where 
\begin{align}
	& J({\bf x}, \lambda) \coloneqq D_{\bf x}r({\bf x}, \lambda) 
	= - \diag({\bf x})D \Phi({\bf x}) - \diag(\Phi({\bf x})) + \lambda I,
	\label{eq-J-def} 
	\\ & \nabla c({\bf x}) 
	= c({\bf x})(
	\|{\bf x}_{1}\|_{p_{1}}^{-p_1} {\bf x}_{1}^{[p_1-1]},\cdots, \|{\bf x}_{d}\|_{p_{d}}^{-p_d} {\bf x}_{d}^{[p_d-1]}
	);\label{eq-nabla-c}
\end{align}
here $\Phi$, $c$ and $r$ are defined in (\ref{eq-Phi-def}) and \eqref{eq-r-def}, respectively.

	The following lemma ensures that the linear system \eqref{eq-nt} is globally well-defined whenever ${\bf x}^{k}>0$ and $\lambda_{k} = \phi({\bf x}^{k})$.
	This result generalizes \cite[Lemma~3]{LGL17}.
	\begin{lemma}\label{nonsingular}
		Consider \eqref{eq-def-H} and \eqref{opt-phi-c=1}.
		Suppose that either Assumption~\ref{ass-str-nonneg} or Assumption~\ref{ass-weak-irr} holds. Let $J$ be defined by \eqref{eq-J-def}, ${\bf x} \in {\mathbb R}_{++}^{{\boldsymbol \sigma}}$ and $\lambda = \phi({\bf x})$. Then, $r({\bf x}, \lambda) \ge 0$ and 
		$DH({\bf x}, \lambda)$ in \eqref{eq-nabla-H} is nonsingular.
		Moreover, 
		\begin{enumerate}[{\rm (i)}]
			\item If either Assumption~\ref{ass-str-nonneg} holds 
			or $r({\bf x}, \lambda)\neq 0$, 
			then $J({\bf x}, \lambda)$ is a nonsingular M-matrix.
			
			\item If Assumption~\ref{ass-weak-irr} holds 
			and $r({\bf x}, \lambda) = 0$, 
			then $J({\bf x}, \lambda)$ is an irreducible singular M-matrix.
		\end{enumerate}
	\end{lemma}
	\begin{proof}
		Let ${\bf x} \in {\mathbb R}_{++}^{{\boldsymbol \sigma}}$
		and $\lambda = \phi({\bf x})$. 
		Since $-D\phi({\bf x})$ is a Z-matrix (see Lemma~\ref{lm-Phi}),
		it follows from \eqref{eq-J-def} that $J({\bf x}, \lambda)$ is also a Z-matrix. From (\ref{eq-DPhi-px}) and \eqref{eq-r-def},
		\begin{equation}\label{eq-J-p-x}
			\begin{split}
				J({\bf x}, \lambda) ({\bf  p}^{[-1]}\otimes {\bf x}) 
				& = - \diag({\bf x})D \Phi({\bf x}) ({\bf  p}^{[-1]}\otimes {\bf x}) + {\bf  p}^{[-1]}\otimes (-\Phi({\bf x})\circ {\bf x} + \lambda {\bf x})
				\\& \textstyle = (1 - \sum_{i=1}^{d}|\sigma_i|/p_i) {\bf  p}^{[-1]}\otimes \Phi({\bf x}) \circ {\bf x} + {\bf  p}^{[-1]}\otimes r({\bf x}, \lambda).
			\end{split}
		\end{equation}
		Furthermore, the definitions of $\phi$ in \eqref{opt-phi-c=1} and $r$ in \eqref{eq-r-def} with $\lambda = \phi(\bf x)$ directly follow that $r({\bf x}, \lambda) \ge 0$.
		
		Now, we consider items (i) and (ii).
		Suppose that Assumption~\ref{ass-str-nonneg} holds.
		Then, we have $\sum_{i=1}^{d}|\sigma_i|/p_i<1$, 
		which, together with $\Phi({\bf x})>0$ (see Lemma~\ref{lm-Phi}(i)), $r({\bf x}, \lambda) \ge 0$ and (\ref{eq-J-p-x}),
		implies that $J({\bf x}, \lambda) ({\bf  p}^{[-1]}\otimes {\bf x})>0$.
		By Lemma~\ref{lm-ZM}, we conclude that $J({\bf x}, \lambda)$ is a nonsingular $M$-matrix.
		
		Suppose that Assumption~\ref{ass-weak-irr} holds.
		Then $-D\phi({\bf x})$ is an irreducible singular $M$-matrix (see Lemma~\ref{lm-Phi}(iii)).
		If $r({\bf x}, \lambda) \neq 0$, by \eqref{eq-J-p-x} we can obtain 
		$$
		J({\bf x}, \lambda) ({\bf p}^{[-1]}\otimes {\bf x}) 
		= {\bf  p}^{[-1]}\otimes r({\bf x}, \lambda) \in \Rsigp\setminus \{0\},
		$$
		which, together with Lemma~\ref{lm-irr-nonsing-M}, implies that
		$J({\bf x}, \lambda)$ is a nonsingular $M$-matrix.
		On the other hand, 	if $r({\bf x}, \lambda) = 0$, then \eqref{eq-J-def} reduces to $J({\bf x}, \lambda)=- \diag({\bf x})D \Phi({\bf x})$, which is an irreducible singular $M$-matrix.
		
		To show the nonsingularity of $DH({\bf x}, \lambda)$ in \eqref{eq-nabla-H}, consider the linear system:
		\begin{equation}
			\label{eq-J-x-nt}
			\begin{split}
				J({\bf x}, \lambda)^{\top} {\bf z} 
				+ \delta \nabla c({\bf x}) \overset{\mathrm{(a)}}{=} 0 
				\quad \mbox{and} \quad 
				{\bf x}^{\top} {\bf z} \overset{\mathrm{(b)}}{=} 0.
			\end{split}
		\end{equation}
		We claim that ${\bf z} =0$ and $\delta = 0$, which completes the proof.
		
		{\bf Case 1}. Either Assumption~\ref{ass-str-nonneg} holds or $r({\bf x}, \lambda)\neq 0$.
		In this case, by item (i) we have that $J({\bf x}, \lambda)$ is a nonsingular M-matrix.
		Left-multiplying both sides of (a) in \eqref{eq-J-x-nt} by ${\bf x}^{\top} (J({\bf x}, \lambda)^{\top})^{-1}$,
		and using (b) in \eqref{eq-J-x-nt},
		we obtain 
		$$
		\delta {\bf x}^{\top} (J({\bf x}, \lambda)^{\top})^{-1} \nabla c({\bf x})
		= 0.
		$$
		From ${\bf x}>0$, $\nabla c({\bf x})>0$ (see \eqref{eq-nabla-c})
		and $(J({\bf x}, \lambda)^{\top})^{-1}\ge 0$ (see Lemma~\ref{lm-ZM}),
		it follows that ${\bf x}^{\top} (J({\bf x}, \lambda)^{\top})^{-1} \nabla c({\bf x}) \neq 0$,
		which implies that $\delta =0$.
		Combining this with (a) in \eqref{eq-J-x-nt} and the nonsingularity of $J({\bf x}, \lambda)$, we conclude that ${\bf z} = 0$.
		
		{\bf Case 2}. Assumption~\ref{ass-weak-irr} holds 
		and $r({\bf x}, \lambda) = 0$.
		In this case, \eqref{eq-J-p-x} reduces to $J({\bf x}, \lambda) ({\bf  p}^{[-1]}\otimes {\bf x})=0$.
		Left-multiplying both sides of (a) in \eqref{eq-J-x-nt} by $({\bf  p}^{[-1]}\otimes {\bf x})^{\top}$, then $\delta ({\bf  p}^{[-1]}\otimes {\bf x})^{\top} \nabla c({\bf x}) = 0$.
		From ${\bf x}>0$ and $\nabla c({\bf x})>0$, it follows that $\delta = 0$.
		Moreover, item (ii) shows that 
		$J({\bf x}, \lambda)$ is an irreducible singular M-matrix.
		Using this to (a) in \eqref{eq-J-x-nt} with $\delta = 0$
		and invoking Lemma~\ref{lm-irr-singular-M},
		we have that ${\bf z} \ge 0$ or ${\bf z} \le 0$.
		Together with (b) in \eqref{eq-J-x-nt} and ${\bf x}>0$,
		this further proves that ${\bf z} = 0$.
	\end{proof}

	Now, we are going to claim that the line search step in \eqref{ls} is also well-defined.
	\begin{lemma}\label{glo}
		Consider \eqref{eq-def-H} and \eqref{opt-phi-c=1}.
		Suppose that either Assumption~\ref{ass-str-nonneg} or Assumption~\ref{ass-weak-irr} holds. 
		Let $R$ be defined by \eqref{eq-R-def}.
		Let ${\bf x}^k\in C$, and let $({\bf d}^k, \delta_k)$ be the solution of (\ref{eq-nt}) with $\lambda_k = \phi({\bf x}^k)$. 
		Then, it holds that
		\begin{equation}\label{eq-mv-phi}
			\phi(R({\bf x}^k + \alpha {\bf d}^k)) = \phi({\bf x}^{k}) + \alpha\delta_{k} + O(\|\alpha{\bf d}^k\|^2) \quad \mbox{as}\ \alpha \to 0^{+}.
		\end{equation}
		Moreover, if $r({\bf x}_k, \lambda_{k}) = 0$, then ${\bf d}^{k}=0$ and $\delta_k = 0$;
		if $r({\bf x}_k, \lambda_{k}) \neq 0$, then
		\begin{equation}\label{eq-delta-r}
			\delta_{k} = -\frac{\nabla c({\bf x}^{k})^{\top}J({\bf x}^{k}, \lambda_{k})^{-1}
			r({\bf x}_k, \lambda_{k})}
			{\nabla c({\bf x}^{k})^{\top}J({\bf x}^{k}, \lambda_{k})^{-1}{\bf x}^{k}}<0,
		\end{equation}
		where $J$ and $\nabla c$ are given in \eqref{eq-J-def} and \eqref{eq-nabla-c}, respectively.
	\end{lemma}
	\begin{remark}
		If $r({\bf x}_k, \lambda_{k}) \neq 0$, then one can 
		see from \eqref{eq-mv-phi} and \eqref{eq-delta-r} 
		that $\phi(R({\bf x}^k + \alpha {\bf d}^k)) \le \phi({\bf x}^{k}) + \alpha \sigma \delta_{k}$
		for all sufficiently small $\alpha>0$
		with constant $\sigma\in (0, 1)$.
	\end{remark}
	\begin{proof}
		For clarity, let
		\begin{equation}
			\label{eq-rk-Jk}
			{\bf r}^k := r({\bf x}^{k}, \lambda_{k})
			\quad \mbox{and} \quad
			J_k := J({\bf x}^{k}, \lambda_{k}),
		\end{equation}
		then ${\bf r}^k \ge 0$ by Lemma~\ref{nonsingular}.
		Thus, from (\ref{eq-nt}), it follows that
		\begin{equation}
			\label{eq-nt-J-nabla-c}
			J_k {\bf d}^{k} + \delta_{k}{\bf x}^{k} + {\bf r}^k = 0\quad\text{and}\quad \nabla c({\bf x}^{k})^{\top}{\bf d}^k = 0.
		\end{equation}
		
		If ${\bf r}^{k} =0$, then from the nonsingularity of $DH({\bf x}^{k}, \lambda_k)$ and ${\bf x}^{k} \in C$, we have ${\bf d}^{k} = 0$ and $\delta_k = 0$.
		
		Suppose that ${\bf r}^k \neq 0$.
		By Lemma~\ref{nonsingular}(i), $J_k$ is a nonsingular $M$-matrix.
		Since ${\bf x}^{k}>0$, $\nabla c({\bf x}^{k})>0$,
		${\bf r}^k \in \Rsigp\setminus\{0\}$
		and $J_{k}^{-1} \ge 0$,
		we can obtain that $\nabla c({\bf x}^{k})^{\top}J_k^{-1}{\bf x}^{k}>0$
		and $\nabla c({\bf x}^{k})^{\top}J_k^{-1}{\bf r}^{k}>0$.
		Left-multiplying both sides of the first equality in \eqref{eq-nt-J-nabla-c} by $\nabla c({\bf x}^{k})^{\top}J_{k}^{-1}$, then, we conclude \eqref{eq-delta-r}.
		
		Next, we turn to show \eqref{eq-mv-phi}.
		From \eqref{eq-J-def},
			\begin{align*}
				D\Phi({\bf x}^k){\bf d}^k  
				&=  \diag({\bf x}^k)^{-1} \left( -J({\bf x}^{k}, \lambda_k) - \diag(\Phi({\bf x}^{k})) + \lambda_{k} I \right){\bf d}^k
				\\& = \diag({\bf x}^k)^{-1}(\delta_{k}{\bf x}^{k} + {\bf r}^k) + ({\bf x}^k)^{[-2]}\circ {\bf r}^k \circ {\bf d}^k
				\\ & = \delta_{k} {\bf 1} + ({\bf x}^k + {\bf d}^k)\circ{\bf r}^k \circ ({\bf x}^k)^{[-2]},
			\end{align*}
		where the second equality holds because of \eqref{eq-nt-J-nabla-c}
		and the definition of $r$ in \eqref{eq-r-def}.
		$\Phi(R(\cdot))$ is twice continuously differentiable 
		in a neighborhood of ${\bf x}^{k}$, then
		\begin{equation*}
			\begin{split}
				& \Phi\left(R({\bf x}^k + \alpha {\bf d}^k)\right) 
				= \Phi({\bf x}^k) + \alpha D\Phi({\bf x}^k) D R({\bf x}_k){\bf d}^k + O(\|\alpha {\bf d}^k\|^2)\\
				& = \Phi({\bf x}^k) + \alpha D\Phi({\bf x}^k){\bf d}^k + O(\|\alpha {\bf d}^k\|^2)
				\\ & =  \Phi({\bf x}^k) + \alpha \delta_{k} {\bf 1}
				+ \alpha ({\bf x}^k +  {\bf d}^k)\circ {\bf r}^{k} \circ ({\bf x}^k)^{[-2]}
				+ O(\|\alpha {\bf d}^k\|^2)
				 \quad \mbox{as}\ \alpha \to 0^{+},
			\end{split}
		\end{equation*}
		where we use \eqref{eq-DR-def} and the fact $\nabla c({\bf x}^{k})^{\top}{\bf d}^{k} = 0$ to get the second equality. Thus,
		\begin{align}
			& \Phi\left(R({\bf x}^k + \alpha {\bf d}^k)\right)
			- (\lambda_{k} + \alpha \delta_{k}) {\bf 1}\notag
			\\& = \Phi({\bf x}^k) - \lambda_{k} {\bf 1}
			+ \alpha ({\bf x}^k + {\bf d}^k)
			\circ {\bf r}^{k} \circ ({\bf x}^k)^{[-2]}
			+ O(\|\alpha {\bf d}^k\|^2)\notag
			\\&= - 
			\left( {\bf x}^{k}  - \alpha ({\bf x}^k + {\bf d}^k)\right)
			\circ {\bf r}^{k}\circ ({\bf x}^{k})^{[-2]}
			+ O(\|\alpha {\bf d}^k\|^2)
			 \quad \mbox{as}\ \alpha \to 0^{+},\label{eq-Phi-delta}
		\end{align}
		where the last equality holds because of \eqref{eq-r-def}
		and \eqref{eq-rk-Jk}.
		
		From the definition of $r$ in \eqref{eq-r-def} and the fact $\lambda_{k} = \phi({\bf x}^{k})$, we have ${\bf r}^{k} \ge 0$
		and $({\bf r}^{k})_{\bar i,\bar{j_i}} = 0$ for some 
		$\bar i\in [d],\bar{j_i}\in [n_{\bar i}]$.
		Since ${\bf x}^{k}>0$, for all sufficiently small $\alpha>0$, we have ${\bf x}^{k}  - \alpha ({\bf x}^k + {\bf d}^k)> 0$,
		which implies
		\begin{equation}\label{eq-miin}
			\textstyle
			\min\limits_{i\in [d], j_i\in [n_i]} \left(\left( {\bf x}^{k}  - \alpha ({\bf x}^k + {\bf d}^k)\right)
			\circ {\bf r}^{k}\circ ({\bf x}^{k})^{[-2]}\right)_{i,j_i} = 0.
		\end{equation}
		By \eqref{eq-Phi-delta} and (\ref{eq-miin}), we conclude that
		\begin{align*}
			& \textstyle \phi(R({\bf x}^k + \alpha {\bf d}^k)) 
			= \max\limits_{i\in [d], j_i\in [n_i]} 
			\Phi_{i,j_i} \left(R({\bf x}^k + \alpha {\bf d}^k)\right) 
			= \lambda_k + \alpha \delta_k + O(\|\alpha {\bf d}^k\|^2)
		\end{align*}
		when $\alpha \to 0^{+}$. This completes the proof.
	\end{proof}
	
	Next, we are going to establish the global convergence of LS-NNM.
	\begin{lemma}\label{lm-bdd}
		Consider \eqref{eq-def-H} and \eqref{opt-phi-c=1}.
		Suppose that either Assumption~\ref{ass-str-nonneg} or Assumption~\ref{ass-weak-irr} holds.
		Let $\{{\bf x}^{k}\}$ and $\{\lambda_{k}\}$ be generated by Algorithm \ref{alg-H-y}. 
		Then, the following statements hold:
		\begin{enumerate}[{\rm (i)}]
			\item 
			There exist constants $\eta>0$ and $M>0$ such that
			\begin{equation}\label{eq-xk-pos-lbd}
				\{{\bf x}^{k}\} \subset C \cap X
				\quad \mbox{with} \quad
				X \coloneqq \{{\bf x} \in \Rsig :  \eta {\bf 1} \le {\bf x} \le M {\bf 1} \}.
			\end{equation}
			\item $\{\lambda_{k}\}$ is nonincreasing and bounded below 
			by $r^{(\bssig, {\bf p})}(\mathcal{A})$.
			\item $\lim_{k \rightarrow \infty}\alpha_k \delta_{k} = 0$.
			\item $\{{\bf d}^{k}\}$ is bounded.
			
		\end{enumerate}
	\end{lemma}
	
	\begin{proof}
		It is straightforward to see from Algorithm~\ref{alg-H-y} that
		${\bf x}^{k} \in C$ and $\lambda_{k} = \phi({\bf x}^{k})$
		for all $k\ge 0.$
		Then, from Theorem~\ref{thr-minmax} and \eqref{ls}, 
		\begin{equation}
			\label{eq-r-lambda}
			r^{(\bssig, {\bf p})}(\mathcal{A})
			\le \lambda_{k+1} \le \lambda_k + \sigma \alpha_k \delta_k 
			\le \lambda_k,
			\quad \forall k=0,1,\ldots,
		\end{equation}
		where the last inequality holds since $\delta_k \le 0$
		(see Lemma~\ref{glo}).
		This shows item (ii).
		
		Moreover, from the above facts, it follows that
		$
		\{{\bf x}^{k}\} \subset \{{\bf x} \in C : 
		\phi({\bf x}) \le \phi({\bf x}^0)\},
		$
		which, together with Proposition~\ref{pro-level-bdd},
		implies item (i).
		
		By \eqref{eq-r-lambda}, $\sigma \delta_k \le \lambda_{k} - \lambda_{k+1}$ holds for every $k \in \mathbb{N}_{0}$.
		Summing these inequalities from $k=0$ to $\infty$,
		and using item (ii), we have $\sum_{k=0}^{\infty} \sigma \delta_k < \lambda_0 - r^{(\bssig, {\bf p})}(\mathcal{A})<\infty$.
		This implies item (iii).
		
		Finally, from \eqref{eq-nt} and $c({\bf x}^{k}) = 1$, 
		it follows that for all $k = 0,1,\ldots$,
		\begin{equation*}
			\begin{split}
				\|{\bf d}^k\|  & \le \left\|
				\begin{pmatrix}
					{\bf d}^k\\
					\delta_k
				\end{pmatrix}
				\right\| 
				= \left\|
				D H({\bf x}^k, \lambda_k)^{-1}
				\begin{pmatrix}
					r({\bf x}^k, \lambda_k)\\
					0
				\end{pmatrix}
				\right\| 
				\le \|
				D H({\bf x}^k, \lambda_k)^{-1}\| \|r({\bf x}^k, \lambda_k)\|.
			\end{split}
		\end{equation*}	
		Note that, on the compact set
		$X \subset \Rsigpp$,
		$D H(\cdot, \phi(\cdot))$ is nonsingular and continuous (see Lemma~\ref{nonsingular}),
		and $r(\cdot, \phi(\cdot))$ is continuous.
		Since $\{{\bf x}^{k}\} \subset X$ (see item(i)),
		$\{\|D H({\bf x}^k, \lambda_k)^{-1}\|\}$ and $\{ \|r({\bf x}^k, \lambda_k)\|\}$ are both bounded. Then, item (iv) holds.
	\end{proof}	
	\begin{theorem}[Global convergence]\label{th-global-converge}
		Suppose that either Assumption~\ref{ass-str-nonneg} or Assumption~\ref{ass-weak-irr} holds.
		Then, the sequence $\{(\lambda_{k}, {\bf x}^{k})\}$, 
		generated by Algorithm \ref{alg-H-y},
		converges to the unique positive $(\bssig, {\bf p})$-eigenpair of $\mathcal{A}$.
	\end{theorem}
	\begin{proof}
		If there exists a nonnegative integer $k$ such that $r({\bf x}^{k}, \lambda_k) = 0$, then Algorithm~\ref{alg-H-y} terminates in finite steps, and $(\lambda_k, {\bf x}^{k})$ is the unique positive $(\bssig, {\bf p})$-eigenpair.
		
		Suppose that $r({\bf x}^{k}, \lambda_k) \neq 0$ for all $k\in \mathbb{N}_{0}$.
		From Lemma~\ref{lm-bdd}(ii), we have $\lim_{k \rightarrow \infty} \lambda_{k} = \lambda^* 
		\ge r^{\bssig, {\bf p}}(\mathcal{A})$.
		By Lemma~\ref{lm-bdd}(i), $\{{\bf x}^{k}\}$ is bounded.
		We claim that any accumulation point of 
		$\{{\bf x}^{k}\}$ is a positive $(\bssig, {\bf p})$-eigenvector of $\mathcal{A}$.
		Since the positive eigenpair of $\mathcal{A}$ is unique, 
		we then conclude the theorem.
		
		Now, suppose on the contrary that there exists an accumulation point ${\bf x}^*$ of 
		$\{{\bf x}^{k}\}$ such that $r({\bf x}^*, \lambda^*) \neq 0$, where ${\bf x}^*>0$ (see \eqref{eq-xk-pos-lbd}).
		Since $\lambda^* = \phi({\bf x}^*)$, by Lemma~\ref{nonsingular}(i), $J^* = J({\bf x}^*, \lambda^*)$ is a nonsingular $M$-matrix.
		This, together with $\nabla c({\bf x}^*)>0$ (see \eqref{eq-nabla-c}), 
		implies that $\nabla c({\bf x}^*)^{\top}(J^*)^{-1}>0$.
		Taking $t \to \infty$ in \eqref{eq-delta-r} with $k=k_t$ yields
		\begin{equation}\label{eq-delta*}
			\lim_{t\rightarrow \infty}\delta_{k_t} 
			= - \frac{\nabla c({\bf x}^*)^{\top}(J^*)^{-1} r({\bf x}^*, \lambda^*)}
			{\nabla c({\bf x}^*)^{\top}(J^*)^{-1}{\bf x}^*} \eqqcolon \delta^* <0,
		\end{equation}
		where the inequality holds because $r({\bf x}^*, \lambda^*)\in \Rsigp \setminus\{0\}$. 
		This and Lemma~\ref{lm-bdd}(iii) yield
		$\lim_{t\rightarrow \infty}\alpha_{k_t} = 0$. 
		
		Thus, there exists an integer $N_1>0$ such that
		condition (\ref{ls}), with $k$ and $\alpha_k$ substituted respectively
		by $k_t$ and $\rho^{-1}\alpha_{k_t}$,
		fails for all integers $t\ge N_1$.
		Furthermore, since $\{{\bf d}^{k}\}$ is bounded,
		it follows from \eqref{eq-xk-pos-lbd} that 
		${\bf x}^{k_t} + \rho^{-1}\alpha_{k_t} {\bf d}^{k_t} > 0
		$ holds for all $t\ge N_2$ with some integer $N_2>0$.
		Thus, 
		$
		\phi( R({\bf x}^{k_t} + \rho^{-1}\alpha_{k_t} {\bf d}^{k_t}) ) - \phi({\bf x}^{k_t}) >  \sigma \rho^{-1}\alpha_{k_t} \delta_{{k_t}}
		$
		for all $t\ge \max\{N_1, N_2\}$.
		
		One can see from Lemma~\ref{lm-bdd}(i) and (iv) that,
		\eqref{eq-mv-phi} holds uniformly for $\{{\bf x}^{k}\}$.
		Dividing both sides of the last display by $\rho^{-1}\alpha_{k_t}$, 
		and taking $t\rightarrow \infty$,
		we obtain from \eqref{eq-mv-phi} that $\delta^* \ge \sigma \delta^*$,
		which contradicts \eqref{eq-delta*}. 
		This completes the proof.
	\end{proof}
	
	\subsection{Quadratic convergence}\label{sec-qc}
	In this subsection, we establish the quadratic convergence of LS-NNM. Our analysis begins with two key lemmas.
		\begin{lemma}\label{lm-phi}
		Consider \eqref{eq-def-H} and \eqref{opt-phi-c=1}.
		Suppose that either Assumption~\ref{ass-str-nonneg} 
		or Assumption~\ref{ass-weak-irr} holds.
		Let ${\bf {\bar x}}$ be a positive 
		$({\boldsymbol \sigma},{\bf  p})$-eigenvector of ${\cal A}$. 
		Then, there exist constants ${\gamma_1}>0$, $c_1>0$ and $c_2>0$ such that 
		\begin{equation}\label{x-lam-equal}
			c_1 \|{\bf x} - {\bf {\bar x}}\| \le \phi({\bf x}) - \phi({\bf {\bar x}}) \le c_2 \|{\bf x} - {\bf {\bar x}}\|,\quad \forall{\bf x} \in \mathbb{B}({\bf {\bar x}}; {\gamma_1}) \cap C.
		\end{equation}
	\end{lemma}
	
	\begin{proof}
		Since $\Phi$ in \eqref{eq-Phi-def} is continuously differentiable 
		at ${\bf \bar{x}}>0$,
		there exist constants $c_2>0$ and ${\rho}_1>0$ such that
		$
		\|\Phi({\bf x}) - \Phi({\bf {\bar x}})\| \le c_2 \|{\bf x} - {\bf {\bar x}}\|
		$
		for all ${\bf x} \in \mathbb{B}({\bf {\bar x}}; {\rho}_1) \cap C$.
		Since ${\bf \bar x}$ is a $(\bssig, {\bf p})$-eigenvector, we have
		$\Phi_{i,j_i}({\bf \bar x}) \equiv \phi({\bf \bar x})$ for all $i\in[d], j_{i}\in [n_i]$. Then, for all ${\bf x} \in \mathbb{B}({\bf {\bar x}}; {\rho}_1) \cap C$,
		\begin{equation*}
			\begin{split}
				\phi({\bf x}) - \phi({\bf \bar x}) 
				& = \max_{i\in [d],j_i\in [n_i]}
				\left(\Phi_{i,j_i}({\bf x}) - \Phi_{i,j_i}({\bf \bar x})\right) 
				\le \|\Phi({\bf x}) - \Phi({\bf {\bar x}})\| \le c_2 \|{\bf x} - {\bf {\bar x}}\|.
			\end{split}
		\end{equation*}
		
		Next, we claim that there exist constants $a>0$ and ${\rho}_2>0$ such that
		\begin{equation}\label{eq-c3}
			\max_{i\in [d],j_i\in [n_i]}
			\left(\nabla\Phi_{i,j_i}({\bf {\bar x}})^{\top}
			\frac{{\bf x} - {\bf {\bar x}}}{ \|{\bf x} - {\bf {\bar x}}\|}\right) \ge a, 
			\quad \forall {\bf x} \in \mathbb{B}({\bf {\bar x}}; {\rho}_2) 
			\cap C \setminus \{ {\bf \bar{x}}\}.
		\end{equation}
		Suppose on the contrary that there exists a sequence $\{{\bf x}^k\}\subset C\setminus\{{\bf {\bar x}}\}$  such that $\lim_{k\to\infty}{\bf x}^k = {\bf {\bar x}}$,
		$
		\max_{i\in [d],j_i\in [n_i]}
		\nabla\Phi_{i,j_i}({\bf {\bar x}})^{\top} {\bf u}^{k} \le \frac{1}{1+k}
		$
		with
		${\bf u}^k \coloneqq ({\bf x}^k - {\bf {\bar x}})/ \|{\bf x}^k - {\bf {\bar x}}\|$ 
		for all $k\ge 0.$
		Since $\{{\bf u}^k\}$ is bounded, without loss of generality, 
		by passing a subsequence if necessary, 
		we assume that $\lim_{k\rightarrow\infty} {\bf u}^k = {\bf \bar u}$.
		Then,
		\begin{equation}\label{eq-ubar}
			\|{\bf {\bar u}}\| \overset{\mathrm{(a)}}{=} 1 
			\quad \mbox{and} \quad
			-D\Phi({\bf {\bar x}}) {\bf \bar u} \overset{\mathrm{(b)}}{\ge} 0.
		\end{equation}
		
		Since $c$ given in (\ref{eq-Phi-def}) is continuously differentiable 
		in a neighborhood of ${\bf \bar x}$, then
		\begin{equation*}
			0 = c({\bf x}^{k}) - c({\bf \bar x}) 
			= \nabla c({\bf \bar{x}})^{\top}({\bf x}^{k} - {\bf \bar x}) 
			+ o(\|{\bf x}^{k} - {\bf \bar x}\|) 
			\quad \mbox{as} \quad k\to \infty.
		\end{equation*}
		Dividing both sides of the above equality 
		by $\|{\bf x}^{k} - {\bf \bar x}\|$, 
		and taking $k\to \infty$, 
		we obtain that $\nabla c({\bf \bar{x}})^{\top}{\bf \bar u} = 0$.
		
		If Assumption~\ref{ass-str-nonneg} holds, by Lemma~\ref{lm-Phi}(ii), $-D\Phi({\bf {\bar x}})$ is a nonsingular M-matrix.
		Combining this with (b) in \eqref{eq-ubar}, we have ${\bf {\bar u}}\ge 0.$
		If Assumption~\ref{ass-weak-irr} holds, by Lemma~\ref{lm-Phi}(iii), $-D\Phi({\bf {\bar x}})$ is a singular irreducible M-matrix.
		Invoking Lemma \ref{lm-irr-nonsing-M}, we can obtain that ${\bf {\bar u}}\ge 0$ or ${\bf {\bar u}}\le 0$.
		In both cases, the fact $\nabla c({\bf \bar{x}})^{\top}{\bf \bar u} = 0$ and $\nabla c({\bf \bar{x}})>0$ (see \eqref{eq-nabla-c}) leads to ${\bf {\bar u}} = 0$, which contradicts (a) in \eqref{eq-ubar}.
		
		Finally, since $\Phi$ is continuously differentiable at ${\bf \bar{x}}$,
		there exists a constant ${\rho}_3>0$ such that
		for all $i\in [d],j_i\in [n_i]$
		\begin{equation*}
			\Phi_{i,j_i}({\bf x}) - \Phi_{i,j_i}({\bf \bar x})
			\ge \nabla \Phi_{i, j_i}({\bf \bar x})^{\top} 
			({\bf x} - {\bf \bar x}) - \tfrac{a}{2} \| {\bf x} - {\bf \bar x}\|,
			\quad \forall {\bf x} \in \mathbb{B}({\bf {\bar x}}; {\rho}_3) 
			\cap C.
		\end{equation*}
		This, together with \eqref{eq-c3}, implies that for all
		${\bf x} \in \mathbb{B}({\bf {\bar x}}; \min\{{\rho}_2, {\rho}_3\}) 
		\cap C$,
		\begin{equation*}
			\begin{split}
				\phi({\bf x}) - \phi({\bf {\bar x}})
				&\textstyle = \max\limits_{i\in [d],j_i\in [n_i]}
				\left(\Phi_{i,j_i}({\bf x}) - \Phi_{i,j_i}({\bf \bar x})\right)
				\ge \tfrac{a}{2} \|{\bf x} - {\bf {\bar x}}\|.
			\end{split}
		\end{equation*}
		Set $\gamma_1 = \min\{{\rho}_1, {\rho}_2, {\rho}_3\}$ 
		and $c_1 = a/2$, then we conclude the proof.
	\end{proof}
	
		\begin{lemma}\label{lm-quad}
		Consider \eqref{eq-def-H} and \eqref{opt-phi-c=1}.
		Suppose that either Assumption \ref{ass-str-nonneg} 
		or Assumption~\ref{ass-weak-irr} holds.
		Let ${\bf \bar x}$ be the unique positive 
		$({\boldsymbol \sigma},{\bf  p})$-eigenvector of ${\cal A}$
		and let $\bar \lambda = \phi({\bf \bar x})$. 
		Let ${\bf x}^{k} \in C$ and $\lambda_{k} = \phi({\bf x}^{k})$.
		Let $({\bf d}^{k}, \delta_{k})$ be the solution of \eqref{eq-nt}.
		Let $\gamma_1$ be given in Lemma~\ref{lm-phi}.
		Then, there exist constants $\gamma_2 \in (0, \gamma_1]$ and $c_3>0$
		such that $\mathbb{B}({\bf \bar x}; \gamma_2) \subset \Rsigpp$, 
		and if ${\bf x}^{k} \in \mathbb{B}({\bf \bar x}; \gamma_2)\cap C$, then
		${\bf x}^{k} + {\bf d}^{k}, R({\bf x}^{k} + {\bf d}^{k}) \in \mathbb{B}({\bf \bar x}; \gamma_2)$, and
		\begin{alignat*}{2}
			\mathrm{(i)} \  & |\lambda_{k} + \delta_{k} - \bar \lambda|
			\le c_3 \| {\bf x}^{k} - {\bf \bar x}\|^2;
			\quad  
			&\mathrm{(ii)}\  \|R({\bf x}^{k} + {\bf d}^{k}) - {\bf \bar x}\|
			\le c_3 \| {\bf x}^{k} - {\bf \bar x}\|^2; \ \ \
			\\ \mathrm{(iii)}\  & \|R({\bf x}^{k} + {\bf d}^{k}) - {\bf \bar x}\|
			\le \tfrac{1}{2} \| {\bf x}^{k} - {\bf \bar x}\|;
			\qquad 
			&\mathrm{(iv)}\   |\phi(R({\bf x}^{k} + {\bf d}^{k})) - \bar \lambda|
			\le c_3 \| {\bf x}^{k} - {\bf \bar x}\|^2.
		\end{alignat*}
	\end{lemma}
	\begin{proof}
		By Lemmas~\ref{pf-sp} and \ref{lm-eqi-rA1}, 
		$\mathcal{A}$ has a unique positive 
		$(\bssig, {\bf p})$-eigenvector ${\bf \bar x}$
		associated with the $(\bssig, {\bf p})$-eigenvalue $\bar\lambda = \phi({\bf \bar x}) = r^{({\boldsymbol \sigma}, {\bf  p})}({\mathcal A})$.
		Observe that the mapping $H$ in \eqref{eq-def-H}
		is continuously differentiable 
		and $D H$ in \eqref{eq-nabla-H} is Lipschitz continuous
		on a neighborhood of $({\bf {\bar x}}, \bar\lambda)$,
		and that $D H({\bf {\bar x}}, \bar\lambda)$ is nonsingular (see Lemma~\ref{nonsingular}).
		It is well known that, under these conditions, 
		the classical Newton method enjoys a local
		quadratic convergence.
		Specifically, there exist constants $\rho_1>0$
		and $a_1>0$ 
		such that $\mathbb{B}(({\bf \bar x}, \bar{\lambda}); \rho_1) \subset \Rsigpp\times \R_{++}$,
		and, moreover, if $(\tilde{\bf x}^k, \tilde\lambda_{k}) 
		\in \mathbb{B}(({\bf \bar x}, \bar{\lambda});\rho_1)$, 
		then the linear system (\ref{eq-nt}) 
		with $({\bf x}^k,\lambda_k) = (\tilde{\bf x}^k,\tilde\lambda_{k})$
		has a unique solution $(\tilde{\bf d}^k,\tilde\delta_k)$ satisfying
		\begin{equation}
			\label{eq-nt-quad}
			\left\|
			\begin{pmatrix}
				\tilde{\bf x}^k + \tilde{\bf d}^k - {\bf {\bar x}},
				\tilde\lambda_{k} + \tilde\delta_k - \bar\lambda
			\end{pmatrix}
			\right\| \le a_1 \left\|
			\begin{pmatrix}
				\tilde{\bf x}^k  - {\bf {\bar x}},
				\tilde\lambda_{k} - \bar\lambda
			\end{pmatrix}
			\right\|^2.
		\end{equation}
		
		Since $R$ in \eqref{eq-R-def} is twice continuously differentiable 
		on $\mathrm{cl}(\mathbb{B}({\bf \bar x}; \rho_1))$,
		there exists a constant $L_{R}>0$ such that
		\begin{equation}\label{eq-R-tcd}
			\|R({\bf x}) - R({\bf z}) - DR({\bf z})({\bf x} - {\bf z})\|
			\le L_{R}\|{\bf x} - {\bf z}\|^2, \quad \forall {\bf x},{\bf z} 
			\in \mathbb{B}({\bf \bar x}; \rho_1). 
		\end{equation}
		
		Define the constants
		\begin{align*}
			& \beta_1 \coloneqq a_1(1 + c_2)^2,
			\quad \rho_2 \coloneqq \min\{\rho_1, \gamma_1, \tfrac{\rho_1}{(1+c_2)}, \tfrac{1}{2\beta_1}\},
			\\ &  \beta_2 \coloneqq \beta_1 + L_{R}(1 + \beta_1 \rho_2)^2,\quad \gamma_2 \coloneqq \min\{\rho_2, \tfrac{1}{2\beta_2}\}
			\quad \mbox{and} \quad 
			c_3 \coloneqq \max\{\beta_1, \beta_2, c_3\beta_2\},
		\end{align*}
		where $\gamma_1$ and $c_2$ are given in Lemma~\ref{lm-phi}.
		Clearly, $\mathbb{B}({\bf \bar x}; \gamma_2) \subseteq \mathbb{B}({\bf \bar x}; \rho_1) \subset \Rsigpp$.
		
		Now, suppose that ${\bf x}^k \in \mathbb{B}({\bf \bar x}; \gamma_2)\cap C$.
		We claim that items~(i)-(vi) hold. Indeed, 
		\begin{equation*}\label{eq-x-lam-ub-x}
			\|({\bf x}^k - {\bf \bar x}, \lambda_{k} -  \bar{\lambda})\| 
			\le \|{\bf x}^{k} - {\bf \bar x}\| + |\lambda_k - \bar{\lambda}|
			\le (1 + c_2) \|{\bf x}^{k} - {\bf \bar x}\| 
			\le\rho_1,
		\end{equation*}
		where the second inequality holds because 
		of $\gamma_2 \le \gamma_1$ and (\ref{x-lam-equal}),
		and the last one follows from $\gamma_2 \le \rho_1/(1+c_2)$.
		This means $({\bf x}^k, \lambda_{k}) \in \mathbb{B}(({\bf \bar x}, \bar{\lambda});\rho_1).$
		Consequently, \eqref{eq-nt-quad} holds for $({\bf x}^k, \lambda_{k})$
		and $({\bf d}^k, \delta_{k})$, namely
		\begin{equation}
			\begin{split}\label{eq-nt-xk1-lambda1}
				&\left\|
				\begin{pmatrix}
					{\bf x}^k + {\bf d}^k - {\bf {\bar x}},
					\lambda_{k} + \delta_k - \bar\lambda
				\end{pmatrix}
				\right\| 
				\le a_1 \left\|
				\begin{pmatrix}
					{\bf x}^k - {\bf {\bar x}},
					\lambda_{k} - \bar\lambda
				\end{pmatrix}
				\right\|^2
				\\& \le a_1 (1 + c_2)^2\|{\bf x}^k  - {\bf {\bar x}}\|^2
				=\beta_1\|{\bf x}^k  - {\bf {\bar x}}\|^2,
			\end{split}
		\end{equation}
		This implies
			$\|{\bf x}^{k} + {\bf d}^{k} - {\bf \bar x}\|
			\le \beta_1 \| {\bf x}^{k} - {\bf \bar x}\|^2$
			and
			$|\lambda_{k} + \delta_{k} - \bar \lambda| 
			\le \beta_1 \| {\bf x}^{k} - {\bf \bar x}\|^2.$
		From $\beta_1 \le c_3$, item (i) follows.
		
		Next, from \eqref{eq-nt-xk1-lambda1}, ${\bf x}^k \in \mathbb{B}({\bf \bar x}; \gamma_2)$,
		and $\gamma_2 \le \rho_2 \le \frac{1}{2\beta_1}$,
		it follows that
		\begin{equation*}\label{eq-nt-xk1-lambda2}
			\left\|
			\begin{pmatrix}
				{\bf x}^k + {\bf d}^k - {\bf {\bar x}},
				\lambda_{k} + \delta_k - \bar\lambda
			\end{pmatrix}
			\right\|
			\le \beta_1 \gamma_2 \|{\bf x}^k  - {\bf {\bar x}}\| 
			\le \tfrac{1}{2}\|{\bf x}^k  - {\bf {\bar x}}\| 
			\le \tfrac{\gamma_2}{2}.
		\end{equation*}
		This shows ${\bf x}^k + {\bf d}^k \in \mathbb{B}({\bf \bar x}; \gamma_2)$.
		In particular, ${\bf x}^k + {\bf d}^k>0$,
		and hence $R({\bf x}^{k} + {\bf d}^{k}) \in C$ is well-defined.
		
		Furthermore, by the triangle inequality, 
		\begin{equation}
			\label{eq-d-bd}
			\begin{split}
				\|{\bf d}^{k}\| & \le  \|{\bf x}^k - {\bf {\bar x}}\| + \|{\bf x}^k + {\bf d}^k - {\bf {\bar x}}\|
				\le \|{\bf x}^k - {\bf {\bar x}}\| +  \beta_1 \| {\bf x}^{k} - {\bf \bar x}\|^2
				\\& \le (1 + \beta_1 \gamma_2)\|{\bf x}^k - {\bf {\bar x}}\| \le (1 + \beta_1 \rho_2)\|{\bf x}^k - {\bf {\bar x}}\|.
			\end{split}
		\end{equation}
		Using the triangle inequality
		and \eqref{eq-R-tcd} (since ${\bf x}^{k}, {\bf x}^k + {\bf d}^k\in \mathbb{B}({\bf \bar x}; \gamma_2)$ and $\gamma_2\le \rho_1$),
		\begin{equation*}
			\begin{split}
				& \|R({\bf x}^k + {\bf d}^k) - {\bf x}^{k} - {\bf d}^k\|
				\le \|R({\bf x}^k) + DR({\bf x}^{k}){\bf d}^k - {\bf x}^{k} - {\bf d}^k\|
				+ L_{R} \|{\bf d}^{k}\|^2 
				\\ & \overset{\mathrm{(a)}}{=} 
				\|\tfrac{1}{d} {\bf x}^{k} \nabla c({\bf x}^{k})^{\top}{\bf d}^k\|
				+ L_{R}\|{\bf d}^{k}\|^2 
				\overset{\mathrm{(b)}}{=} L_{R}\|{\bf d}^{k}\|^2  
				\overset{\mathrm{(c)}}{\le} L_R (1 + \beta_1 \rho_2)^2 \|{\bf x}^k - {\bf {\bar x}}\|^2,
			\end{split}
		\end{equation*}
		where (a) follows from \eqref{eq-DR-def} and the fact ${\bf x}^k \in C$,
		and (b) holds since $\nabla c({\bf x}^{k})^{\top}{\bf d}^k = 0$
		(see \eqref{eq-nt} with $c({\bf x}^{k}) = 1$),
		and (c) is deduced from \eqref{eq-d-bd}.
		Using the triangle inequality again, we obtain that
		\begin{equation}
			\begin{split}
				\|R({\bf x}^k + {\bf d}^k) - {\bf \bar{x}}\| 
				& \le \|{\bf x}^{k} + {\bf d}^k - {\bf \bar x}\| + \|R({\bf x}^k + {\bf d}^k) - {\bf x}^{k} - {\bf d}^k\|
				\\ & \le (\beta_1 + L_R (1 + \beta_1 \rho_2)^2) \|{\bf x}^k - {\bf {\bar x}}\|^2
				= \beta_2 \|{\bf x}^k - {\bf {\bar x}}\|^2,\label{eq-R-x-rate}
			\end{split}
		\end{equation}
		This, together with $\beta_2 \le c_3$, shows item (ii).
		
		Together with ${\bf x}^k \in \mathbb{B}({\bf \bar x}; \gamma_2)$,
		(\ref{eq-R-x-rate}) also implies that
		\begin{equation*}
			\|R({\bf x}^k + {\bf d}^k) - {\bf \bar{x}}\| 
			\le \beta_2 \gamma_2 \|{\bf x}^k - {\bf {\bar x}}\|
			\le \tfrac{1}{2}\|{\bf x}^k - {\bf {\bar x}}\| \le \tfrac{\gamma_2}{2},
		\end{equation*}
		where $\gamma_2 \le \frac{1}{2\beta_2}$ is used to get the second inequality.
		Then, item (iii) as well as $R({\bf x}^k + {\bf d}^k)\in  \mathbb{B}({\bf \bar x}; \gamma_2)$ hold.
		
		Finally, since $R({\bf x}^k + {\bf d}^k)\in \mathbb{B}({\bf \bar x}; \gamma_2)\cap C$ and
		$\gamma_2\le \gamma_1$, 
		it follows from \eqref{x-lam-equal} that
		\begin{equation*}
			\begin{split}
				0\le \phi(R({\bf x}^k + {\bf d}^k)) - \phi({\bf \bar x})
				\le c_2 \|R({\bf x}^k + {\bf d}^k) - {\bf \bar x}\| 
				\le c_2 \beta_2 \|{\bf x}^k - {\bf {\bar x}}\|^2,
			\end{split}
		\end{equation*}
		Since $c_2 \beta_2 \le c_3$, we conclude item (iv).
	\end{proof}
	
	The local quadratic convergence
	of LS-NNM is presented below.
	
	\begin{lemma}\label{th-local}
		Consider \eqref{eq-def-H} and \eqref{opt-phi-c=1}.
		Suppose that either Assumption \ref{ass-str-nonneg} 
		or Assumption~\ref{ass-weak-irr} holds. 
		Let ${\bf x}^0\in C$ and $\lambda_{0} = \phi({\bf x}^{0})$.
		Consider the iteration scheme:
		$$
		{\bf x}^{k+1}=R({\bf x}^k+{\bf d}^k)
		\quad \mbox{and} \quad
		 \lambda_{k+1}=\phi({\bf x}^{k+1}),
		 \quad \forall k=0, 1, \ldots,
		$$
		where ${\bf d}^{k}$ is given in \eqref{eq-nt}.
		If ${\bf x}^0$ is sufficiently closed to a positive $({\boldsymbol \sigma},{\bf  p})$-eigenvector ${\bf \bar x}$ of ${\cal A}$, 
		then the iteration sequences $\{{\bf x}^{k}\}$ and $\{\lambda_{k}\}$ converge quadratically to ${\bf \bar x}$ and $\bar{\lambda}$, respectively,
		where $\bar{\lambda} \coloneqq \phi({\bf \bar{x}}) = r^{({\boldsymbol \sigma}, {\bf  p})}({\mathcal A})$. 
	\end{lemma}
	\begin{proof}
		Let ${\bf x}^{0} \in \mathbb{B}({\bf \bar x}; \gamma_2)\cap C$,
		where $\gamma_2$ is given in Lemma~\ref{lm-quad}.
		Then, by Lemma~\ref{lm-quad}, one can prove by induction that
		for all $k=0,1,\ldots$,
		$$
		{\bf x}^{k+1} = R({\bf x}^{k} + {\bf d}^{k}) \in \mathbb{B}({\bf \bar x}; \gamma_2) \cap C
		\quad \mbox{and} \quad
		\|{\bf x}^{k+1} - {\bf \bar x}\| 
		\le \tfrac{1}{2} \|{\bf x}^{k} - {\bf \bar x}\|.
		$$
		Iterating the above inequality yields that
		$
		\|{\bf x}^{k+1} - {\bf \bar x}\| \le \tfrac{1}{2^{k+1}}\|{\bf x}^{0} - {\bf \bar x}\|
		\to 0$
		as 
		$k\to \infty.$
		By this and Lemma~\ref{lm-quad}(ii), the sequence 
		$\{{\bf x}^{k}\}$ converges to ${\bf \bar x}$ quadratically.
		
		Moreover, using \eqref{x-lam-equal}, we have
		$|\lambda_{k} - \bar \lambda| = \Theta(\|{\bf x}^{k} - {\bf \bar x}\|)$.
		This implies that $\lim_{k\to \infty} \lambda_{k} = \bar \lambda$,
		and
$	|\lambda_{k+1} - \bar \lambda| = \Theta(\|{\bf x}^{k+1} - {\bf \bar x}\|)
			\le O(\|{\bf x}^{k} - {\bf \bar x}\|^2) = O((\lambda_{k} - \bar \lambda)^2).$
	\end{proof}
	
	Finally, we establish the global quadratic convergence
	of LS-NNM.
	\begin{theorem}[Quadratic convergence]
		Suppose that either Assumption \ref{ass-str-nonneg} or 
		Assumption~\ref{ass-weak-irr} holds.
		Let ${\bf \bar x}$ be the unique $(\bssig, {\bf p})$-eigenvector of $\mathcal{A}$
		and $\bar \lambda = r^{\bssig, {\bf p}}(\mathcal{A})$.
		Then, the sequences $\{{\bf x}^{k}\}$ and $\{\lambda_{k}\}$, generated by Algorithm~\ref{alg-H-y},
        converge quadratically to 
		${\bf \bar x}$ and $\bar{\lambda}$, respectively. 
		Moreover, $\alpha_k = 1$ holds for all sufficiently large $k$.
	\end{theorem}
	\begin{proof}
		We claim that, for all sufficiently large $k$, both conditions in \eqref{ls}
		hold with $\alpha_k = 1$.
		Then, the quadratic convergence follows from
		Theorem~\ref{th-global-converge} and Lemma~\ref{th-local}.
		
		Let $\gamma_1$ and $c_1$ be 
		given in Lemma~\ref{lm-phi}.
		Let $\gamma_2$ and $c_3$ be given in Lemma~\ref{lm-quad}.
		Let $\gamma = \min\{\gamma_1, \gamma_2\}$.
		By Theorem~\ref{th-global-converge}, $\lim_{k\to\infty} ({\bf x}^{k}, \lambda_k) = ({\bf \bar x}, \bar \lambda)$.
		There exists an integer $N_1\ge 0$
		such that ${\bf x}^{k} \in \mathbb{B}({\bar x}; \gamma)$
		for all integers $k\ge N_1$.
		From Lemma~\ref{lm-quad}, we have 
		\begin{equation}\label{eq-x-reag}
			{\bf x}^{k} + {\bf d}^{k}\in \mathbb{B}({\bf \bar x}; \gamma) \subset \Rsigpp
			\quad\mbox{and}\quad 
			R({\bf x}^{k} + {\bf d}^{k}) \in \mathbb{B}({\bf \bar x}; \gamma),
			\quad \forall k\ge N_1.
		\end{equation}
		
		For each integer $k\ge N_1$, 
		we have ${\bf x}^{k} \in \mathbb{B}({\bf \bar x}; \gamma)\cap C$ and
		\begin{equation*}
			\begin{split}
				\delta_{k} & = -(\lambda_k - \bar \lambda) + (\lambda_k + \delta_k - \bar \lambda) 
				\le - c_1 \|{\bf x}^k - {\bf \bar x}\| + c_3\|{\bf x}^k - {\bf \bar x}\|^2,
			\end{split}
		\end{equation*}
		where the inequality holds because of \eqref{x-lam-equal} and Lemma~\ref{lm-quad}(ii).
		Furthermore, 
		\begin{equation*}
			\begin{split}
				|\phi( R({\bf x}^k +  {\bf d}^k) ) - (\lambda_{k} +  \delta_{k})|
				& \le |\phi(R({\bf x}^k +  {\bf d}^k) ) - {\bar{\lambda}}|
				+  |\lambda_{k} + \delta_k - {\bar{\lambda}}|
				 \le 2c_3\|{\bf x}^k - {\bf \bar x}\|^2
			\end{split}
		\end{equation*}
		for all $k\ge N_1,$ where the last inequality follows from Lemma~\ref{lm-quad}(ii) and (iv).
		
		Combining the last two displays, we conclude that
		\begin{equation*}
			\begin{split}
				& \phi( R({\bf x}^k +  {\bf d}^k) ) - (\lambda_{k} + \sigma  \delta_{k})
				= (1-\sigma) \delta_{k} + \phi( R({\bf x}^k +  {\bf d}^k) ) - (\lambda_{k} +  \delta_{k})
				\\ & \le  (1-\sigma) (- c_1 \|{\bf x}^k - {\bf \bar x}\| + c_3\|{\bf x}^k - {\bf \bar x}\|^2) + 2c_3\|{\bf x}^k - {\bf \bar x}\|^2
				\\ & = -(1-\sigma) c_1 \|{\bf x}^k - {\bf \bar x}\| + (3-\sigma)c_3\|{\bf x}^k - {\bf \bar x}\|^2,
				\quad \forall k\ge N_1.
			\end{split}
		\end{equation*}
		Since $\sigma\in (0, 1)$, there exists an integer $N\ge N_1$ such that
		$\phi( R({\bf x}^k +  {\bf d}^k) ) \le \lambda_{k} + \sigma  \delta_{k}$
		for all $k\ge N$.		
		Finally, combining this with \eqref{eq-x-reag}, we conclude that, 
		for all integer $k\ge N$, 
		both conditions in \eqref{ls} hold with $\alpha_k = 1$.
	\end{proof}	
	
\section{Numerical experiments}\label{sec-ne}
In this section, we study the numerical performance of Algorithm~\ref{alg-H-y},
denoted by LS-NNM, for computing the positive
\((\boldsymbol{\sigma},\mathbf{p})\)-eigenpair of nonnegative tensors.
We compare LS-NNM with the general power method (denoted by PM1) proposed in
\cite[p.~514, eq.~(20)]{GTH23}, as well as with its variants (denoted by PM2) given in
\cite[p.~515, eqs.~(21)--(22)]{GTH23}.\footnote{All numerical experiments were
conducted in MATLAB R2022a on a personal laptop equipped with an
Intel(R) Core(TM) i7-10510U CPU @ 1.80\,GHz and 16\,GB of memory,
running Microsoft Windows~11.}\footnote{According to \cite[Theorem~4.3]{GTH23} 
and Lemma~\ref{lm-eqi-rA1}, under Assumption~\ref{ass-str-nonneg},
both PM1 and PM2 converge linearly to the unique positive
\(\sigp\)-eigenpair; under
Assumption~\ref{ass-weak-irr}, PM2 converges to the unique positive
\(\sigp\)-eigenpair, whereas the convergence of PM1 remains unknown.}

For all these three methods, the initial point is set to ${\bf x}^0 = {\bf 1}\in {\mathbb R}^{{\boldsymbol \sigma}}$. 
The stopping criteria are either reaching 500 iterations or satisfying the following condition:  
\begin{equation}\label{res}
	\mathsf{Res} \coloneq \frac{\phi( {\bf {\bar x}}^{k}) - \psi( {\bf {\bar x}}^{k})}{\max\{1, \psi( {\bf {\bar x}}^{k})\}}\le 10^{-12}
    \quad \mbox{with} \quad {\bf {\bar x}}^{k} \coloneq (\|{\bf x}_{1}^{k}\|_{p_1}^{-1} {\bf x}_{1}^{k}, \ldots, \|{\bf x}_{d}^{k}\|_{p_d}^{-1} {\bf x}_{d}^{k}),
\end{equation}
where $\phi$ is given in \eqref{opt-phi-c=1}, and 
$\psi({\bf x}) \coloneqq \min_{i\in [d],j_i\in [n_i]} \Phi_{i, j_i}({\bf x})$ for all ${\bf x} \in 
\Rsigpp$, with $\Phi$ defined in \eqref{eq-Phi-def}.
Note that, for every ${\bf x}^{k} \in \mathbb{R}_{++}^{({{\boldsymbol \sigma}},{\bf p})}$, 
we have from \eqref{maximality-2} that 
$\mathsf{Res}$ provides an upper bound for the relative error: 
\begin{equation}\label{eq-rel-error}
    {| \lambda^* -
r^{({\boldsymbol \sigma},{\bf  p})}({\cal A}) |}/{\max\{1, r^{({\boldsymbol \sigma},{\bf  p})}({\cal A})\}} \le \mathsf{Res},
\end{equation}
where $\lambda^* \coloneqq \tfrac{1}{2} (\phi({\bf {\bar x}}^{k}) + \psi({\bf {\bar x}}^{k}))$.
\begin{remark}[Stopping criterion of LS-NNM]
    Suppose that either Assumption~\ref{ass-str-nonneg} or Assumption~\ref{ass-weak-irr} holds.
    Let $\{{\bf x}^{k}\}$ be generated by Algorithm~\ref{alg-H-y}.
    Then, the condition \eqref{res} can be used as the stopping criterion of LS-NNM since from \eqref{eq-xk-pos-lbd} one can obtain that 
    $\| {\bf {\bar x}}^{k} - {\bf {\bar x}}\| = O (\| {\bf  x}^{k} - {\bf {\bar x}}\|)$,
    where ${\bf {\bar x}}$ is a positive $\sigp$-eigenpair of $\mathcal{A}$. 
\end{remark}

Here, we present an third order three dimensional nonnegative tensor, with different shape partitions.
\begin{example}[{\cite[Example~7.2]{GTH23}}]\label{ex-small-1}
	Tensor $\mathcal{A}=(a_{ijk}) \in  \mathbb{R}_+^{3\times 3\times 3}$ is defined as
	$$
	a_{113} =a_{131} = a_{221} =a_{222} = a_{321} =  1,
	$$
	and $a_{ijk} = 0 $ elsewhere. Let 
	${\boldsymbol \sigma}^1 = \{ \{ 1, 2, 3\} \}$,
	${\boldsymbol \sigma}^2 = \{ \{ 1\}, \{ 2, 3\} \}$, 
    and ${\boldsymbol \sigma}^3 = \{ \{ 1\}, \{ 2\}, \{ 3\} \}$ 
	be three shape partitions of $\mathcal{A}$. 
    Then, $\mathcal{A}$ is ${\boldsymbol \sigma}^1$- and ${\boldsymbol \sigma}^3$-weakly irreducible, and ${\boldsymbol \sigma}^2$-strictly nonnegative, but not ${\boldsymbol \sigma}^2$-weakly irreducible.
\end{example}

The results about this example are listed in Table \ref{tab-small} below.
Here, `iter’, `iter-ls', and `Cpu’ represent iterations, inner iterations, 
and computing time in seconds, respectively; 
$\mathsf{Res}$ and $\lambda^*$ are given as in \eqref{res} and \eqref{eq-rel-error}, respectively.
Moreover, the third column reports whether the condition $\rho(A)<1$ (or $= 1$) in Lemma~\ref{pf-sp} is satisfied,
which follows directly Lemma~\ref{lm-eqi-rA1}.

\begin{center}
	\begin{table}[!htbp]
		\caption{Numerical results for \Cref{ex-small-1}}\label{tab-small}
		\setlength\tabcolsep{2pt}
		{\footnotesize
			\def\temptablewidth{1\textwidth}
			\begin{tabular*}{\temptablewidth}{@{\extracolsep{\fill}}ccccccc}
				\toprule
				\multirow{2}{*}{\tabincell{c}{ ${\boldsymbol \sigma}$ }} 
                & \multirow{2}{*}{${\bf  p}$}
                & $\rho(A)<1$ 
                & \multirow{2}{*}{$\lambda^*$}  & PM1 & PM2 & LS-NNM\\
				& & (or $= 1$) & & \textsf{iter} / \textsf{Cpu} / \textsf{Res} & \textsf{iter} / \textsf{Cpu} / \textsf{Res}  & \textsf{iter} / \textsf{iter-ls} / \textsf{Cpu} / \textsf{Res}\\
				\hline
				\multirow{3}{*}{\tabincell{c}{$\bssig^{1}$}}& (3) & $=$
				& 1.748 &  45 / 0.0007 / 6.1e-13
				&104 / 0.0017 / 9.8e-13
				& 4 /   0 / 0.0005 / 9.1e-14\\
				& (4) & $<$
				& 2.277 &  26 / 0.0005 / 4.0e-13
				& 68 / 0.0013 / 8.2e-13
				& 4 /   0 / 0.0005 / 7.8e-16\\
				& (5) & $<$
				& 2.663 &  20 / 0.0003 / 5.7e-13
				& 58 / 0.0010 / 8.0e-13
				& 4 /   0 / 0.0005 / 6.3e-15\\
				\hline
				\multirow{3}{*}{\tabincell{c}{$\bssig^{2}$}}& (2,4) & $<$
				& 1.414 & 191 / 0.0045 / 9.0e-13
				&395 / 0.0097 / 9.7e-13
				& 8 /   0 / 0.0019 / 1.6e-16\\
				& (3,5) & $<$
				& 2.167 &  38 / 0.0012 / 7.3e-13
				& 92 / 0.0029 / 7.9e-13
				& 5 /   0 / 0.0015 / 2.1e-13\\
				& (4,6)& $<$
				& 2.581 &  26 / 0.0008 / 8.5e-13
				& 69 / 0.0022 / 9.8e-13
				& 5 /   0 / 0.0014 / 8.6e-16\\
				\hline
				\multirow{3}{*}{\tabincell{c}{$\bssig^{3}$}}& (3,3,3) & $=$
				& 2.045 &  72 / 0.0032 / 7.2e-13
				&158 / 0.0073 / 8.9e-13
				& 5 /   0 / 0.0015 / 4.4e-14\\
				& (4,4,4) & $<$
				& 2.469 &  42 / 0.0021 / 5.9e-13
				& 99 / 0.0052 / 8.3e-13
				& 5 /   0 / 0.0016 / 3.6e-16\\
				& (5,5,5) & $<$
				& 2.817 &  30 / 0.0015 / 6.1e-13
				& 76 / 0.0038 / 9.2e-13
				& 5 /   0 / 0.0015 / 7.9e-16\\
				\bottomrule
		\end{tabular*}}
	\end{table}
\end{center}

It can be seen from Table \ref{tab-small} that all of three methods efficiently solve the corresponding $({\boldsymbol \sigma},{\bf  p})$-spectral problem 
across different choices of ${\boldsymbol \sigma}$ and ${\bf  p}$. 
Among them, the proposed LS-NNM requires fewer iterations, achieves comparable CPU time,  
and attains the highest accuracy.
Besides, it is worth noting that, in all cases, the inner iterations of LS-NNM are zero. 
\vspace{-0.4 cm}	
\bibliographystyle{siamplain}
\bibliography{references}
\end{document}